\documentclass[12pt]{amsart}

\usepackage[all]{xypic}
\usepackage{tikz}
\usetikzlibrary{arrows} 
\usetikzlibrary{decorations.markings}
\usepackage{graphicx}
\usepackage{bm}
\usepackage{epsf}
\usepackage{verbatim} 
\usepackage{amsmath}
\usepackage{amsfonts}
\usepackage{amssymb}
\usepackage{mathrsfs}
\usepackage{amsthm}
\usepackage{newlfont}
\usepackage[answerdelayed,lastexercise]{exercise}

\newtheorem{thm}{Theorem}[section]
\newtheorem{prop}[thm]{Proposition}
\newtheorem{lem}[thm]{Lemma}
\newtheorem{cor}[thm]{Corollary}

\theoremstyle{definition}
\newtheorem{defn}[thm]{Definition}

\newtheorem{example}[thm]{Example}

\theoremstyle{remark}
\newtheorem{rem}[thm]{Remark}

\numberwithin{equation}{section}
\numberwithin{thm}{section}

\DeclareMathOperator{\Hom}{Hom}
\DeclareMathOperator{\End}{End}
\DeclareMathOperator{\Aut}{Aut}

\DeclareMathOperator{\Fr}{Fr}
\DeclareMathOperator{\Spec}{Spec}
\DeclareMathOperator{\inv}{inv}
\DeclareMathOperator{\Lie}{Lie}
\DeclareMathOperator{\chr}{char}

\DeclareMathOperator{\Gal}{Gal}
\DeclareMathOperator{\GL}{GL}

\DeclareMathOperator{\Cl}{Cl}

\DeclareMathOperator{\Tr}{Tr}

\DeclareMathOperator{\diag}{diag}
 
\DeclareMathOperator{\Nr}{Nr}

\DeclareMathOperator{\Ram}{Ram}
\DeclareMathOperator{\rG}{G}

\newcommand{\sep}{\mathrm{sep}}
\newcommand{\opp}{\mathrm{opp}}
\newcommand{\alg}{\mathrm{alg}}

\newcommand{\ab}{\mathrm{ab}}

\newcommand{\fL}{\mathfrak{L}}

\newcommand{\fP}{\mathfrak{P}}
\newcommand{\fQ}{\mathfrak{Q}}

\newcommand{\fY}{\mathfrak{Y}}

\newcommand{\fd}{\mathfrak{d}}

\newcommand{\fl}{\mathfrak{l}}
\newcommand{\fm}{\mathfrak{m}}
\newcommand{\fn}{\mathfrak{n}}

\newcommand{\fp}{\mathfrak{p}}
\newcommand{\fq}{\mathfrak{q}}

\newcommand{\fy}{\mathfrak{y}}

\newcommand{\cC}{\mathcal{C}}
\renewcommand{\cD}{\mathcal{D}}

\newcommand{\cN}{\mathcal{N}}

\newcommand{\cP}{\mathcal{P}}

\newcommand{\cT}{\mathcal{T}}

\newcommand{\cW}{\mathcal{W}}

\newcommand{\sD}{\mathscr{D}}

\newcommand{\A}{\mathbb{A}}

\newcommand{\C}{\mathbb{C}}

\newcommand{\F}{\mathbb{F}}
\newcommand{\gm}{\mathbb{G}}

\newcommand{\p}{\mathbb{P}}
\newcommand{\Q}{\mathbb{Q}}

\newcommand{\Z}{\mathbb{Z}}

\newcommand{\tF}{\widetilde{F}}

\newcommand{\tinf}{\widetilde{\infty}}

\newcommand{\eps}{\varepsilon}
\newcommand{\G}{\Gamma}
\newcommand{\To}{\longrightarrow}

\newcommand{\Fi}{F_\infty}
\newcommand{\Ci}{\C_\infty}

\newcommand{\La}{\Lambda}

\newcommand{\Leg}[2]{\genfrac{(}{)}{}{}{#1}{#2}}
\newcommand{\Mod}[1]{\ (\mathrm{mod}\ #1)}

\newcommand{\ls}[2]{#1\left(\!\left(#2\right)\!\right)}
\newcommand{\dvr}[2]{#1\left[\!\left[#2\right]\!\right]}

\begin{document}
\title{Drinfeld-Stuhler modules and the Hasse principle}

\author{Keisuke  Arai} 
\address{Department of Mathematics, 
	School of Science and Technology for Future Life,
	Tokyo Denki University,
	5 Senju Asahi-cho, Adachi-ku, Tokyo 120-8551, Japan}
\email{araik@mail.dendai.ac.jp}

\author{Satoshi Kondo} 
\address{Middle East Technical University,
Northern Cyprus Campus,
99738 Kalkanli,
Guzelyurt,
Mersin 10,
Turkey}

\address{Kavli Institute for the Physics and Mathematics of the Universe,
University of Tokyo,
Kashiwa-no-ha 5-1-5,
277-8583,
Japan}

\email{satoshi.kondo@gmail.com}

\author{Mihran Papikian}
\address{Department of Mathematics, Pennsylvania State University, University Park, PA 16802, U.S.A.}
\email{papikian@psu.edu}



\begin{abstract} We develop a theory of canonical isogeny characters of Drinfeld-Stuhler modules similar to the theory 
	of canonical isogeny characters of abelian surfaces with quaternionic multiplication. 
	We then apply this theory to give explicit criteria for the non-existence of rational points on Drinfeld-Stuhler 
	modular varieties over the finite extensions of $\mathbb{F}_q(T)$. 
	This allows us to produce explicit examples of Drinfeld-Stuhler curves violating the Hasse principle. 
\end{abstract}

\maketitle


\section{Introduction} 

Drinfeld-Shutler modules are function field analogues of abelian surfaces equipped with 
an action of an indefinite quaternion algebra over $\Q$. The idea of these objects
was proposed in the language of shtukas by Ulrich Stuhler under the name of \textit{$\sD$-elliptic sheaf}  
as a natural generalization of Drinfeld modules. 
The modular varieties of $\sD$-elliptic sheaves were studied by Laumon, Rapoport and Stuhler in \cite{LRS}, with the aim of proving the local Langlands correspondence for $\GL(n)$ in positive characteristic. 

Over the years, the third author of this paper have studied the arithmetic properties of $\sD$-elliptic sheaves and their modular varieties, trying to extend  
to this setting the rich theory of abelian surfaces with quaternionic multiplication and 
Shimura curves;  see, for example, \cite{PapGenus}, \cite{PapCrelle1}, \cite{PapHE}, \cite{PapikianMZ}, 
\cite{PapikianCrelle2}, \cite{PapJNTHayes}, 
\cite{PapRMS}. The current paper is a natural continuation of \cite{PapRMS}. 

Let $\F_q$ be a finite field with $q$ elements, where $q$ is a power of a prime number. Let $A=\F_q[T]$ 
be the ring of polynomials in indeterminate $T$ with coefficients in $\F_q$, and $F=\F_q(T)$ be the field of fractions of $A$. 
Let $d\geq 2$ be an integer and $D$ be a central division algebra over $F$ of dimension $d^2$, which is split at $1/T$. 
Fix a maximal $A$-order $O_D$ in $D$. The modular variety of Drinfeld-Stuhler $O_D$-modules $X^D$ 
is a projective geometrically connected  variety of dimension $d-1$ defined over $F$ (see Section \ref{sPreliminaries}). 

This paper has three main objectives:  
\begin{enumerate}
	\item Develop a theory of canonical isogeny characters of Drinfeld-Stuhler modules similar to the theory 
	of canonical isogeny characters of abelian surfaces with quaternionic multiplication given in \cite{Jordan}. 
	\item Apply (1) to give explicit criteria for the non-existence of rational points on $X^D$ over finite extensions of $F$. 
	\item Combine (2) with the results in \cite{PapikianCrelle2} to produce explicit examples of Drinfeld-Stuhler curves 
	violating the Hasse principle. 
\end{enumerate}

The major inspiration for this paper has been the work of Bruce Jordan in \cite{Jordan} 
(which itself was inspired by the work of Mazur \cite{MazurIsog}). One significant difference 
between our work and  \cite{Jordan} is that we are able to carry out (1) and (2) for arbitrary $d\geq 2$, not just quaternion algebras and 
curves. In principle, (3) can also be extended to the higher dimensional Drinfeld-Stuhler varieties once the results in \cite{PapikianCrelle2} 
are extended to the higher dimensional varieties. 
We also mention that the results of Jordan in \cite{Jordan}  have been generalized in the case of Shimura curves by 
the first author in \cite{Arai1}, \cite{Arai2}. 

Next we give a more detailed description of the results in this paper. 
Let $K$ be a field equipped with an $A$-algebra structure $\gamma:A\to K$. 
Note that $\F_q$ is a subfield of $K$.  Let $\tau$ be the $\F_q$-linear Frobenius endomorphism of the 
additive group-scheme $\gm_{a,K}=\Spec(K[x])$ over $K$; the morphism $\tau$ is given on the underlying ring by $x\mapsto x^q$. 
The ring of $\F_q$-linear endomorphisms $\End_{\F_q}(\gm_{a, K})$ is canonically isomorphic to the 
skew polynomial ring $K[\tau]$ with the commutation relation $\tau \alpha=\alpha^q\tau$, $\alpha\in K$. 
Given a unitary ring $R$, let $M_d(R)$ 
denote the ring of $d\times d$ matrices with entries from $R$. 
A Drinfeld-Stuhler $O_D$-module over $K$ is an embedding 
$$\phi: O_D\To \End_{\F_q}(\mathbb{G}_{a, K}^d)\cong M_d(K[\tau]), \qquad b\longmapsto \phi_b,
$$ 
such that: (i) For any $b\neq 0$ the kernel of $\phi_b$ as an endomorphism of $\mathbb{G}_{a, K}^d$ is 
a finite group scheme of order $\# O_D/O_D b$; (ii) For any $a\in A$ substituting $0$ for $\tau$ in $\phi_a$, one obtains 
the scalar matrix $\gamma(a)I_d$. 

Assume $\fp$ is a prime of $A$ such that the Hasse invariant $\inv_\fp(D)$ of $D$ at $\fp$ is $1/d$. 
Denote $\F_\fp=A/\fp$ and let $\F_\fp^{(d)}$ be the degree $d$ extension of the finite field $\F_\fp$. 
Let $\phi$ be a Drinfeld-Stuhler $O_D$-module over $K$. Assume $\fp\neq \ker(\gamma)$. In Section \ref{sCanChar}, 
we show that $\ker\phi_\fp\cong O_D/\fp$ has a canonical subgroup $\cC_{\phi, \fp}\cong \F_\fp^{(d)}$ which is $O_D$-invariant and rational 
over $K$. Following the terminology in \cite{Jordan}, we call 
the associated Galois representation 
$$
\varrho_{\phi, \fp}: \Gal(\overline{K}/K)\to \Aut_{O_D}(\cC_{\phi, \fp}) \approx (\F_\fp^{(d)})^\times
$$
a \textit{canonical isogeny character} of $\phi$ at $\fp$. We then study the properties of $\varrho_{\phi, \fp}$. 
To prove some of these properties, one needs to know that Drinfeld-Stuhler modules have potentially good reduction over local fields. 
This foundational result is implicitly proved in \cite{LRS} in the language of $\sD$-elliptic sheaves, but  
in Section \ref{sPGR} we give a different proof, along with a more precise information about the extension 
over which a Drinfeld-Stuhler module acquires good reduction.  

Now let $K$ be a degree $d$ extension of $F$ which splits $D$. Assume there is a Drinfeld-Stuhler $O_D$-module 
$\phi$ defined over $K$. 
The canonical isogeny character of $\phi$ forces the coefficients of the characteristic polynomials of the Frobenius automorphisms  
acting on a Tate module of $\phi$ to satisfy certain congruences modulo $\fp$ (this is proved in Section \ref{sDSMoverFF}), and 
it also puts some restrictions on the ray class groups of $K$. 
We use this information in Section \ref{sGlobPoints} to show that if 
certain explicit conditions are satisfied, then the Drinfeld-Stuhler variety $X^D$ has no $K$-rational points. 
These conditions are listed in Theorems \ref{thmGlobalPoints}, \ref{thmMain1}, \ref{thmMain2}, which are the main 
results of this paper. 
A technical issue arises here from the fact that 
$X^D$ is only a coarse moduli scheme, so the $K$-rational points on $X^D$ may not be represented by 
Drinfeld-Stuhler $O_D$-modules defined over $K$. Fortunately, a result from \cite{PapRMS} partly resolves this:  
If $X^D(K)\neq \emptyset$, then there is a  Drinfeld-Stuhler $O_D$-module defined over $K$ if and only if $D\otimes_F K=M_d(K)$.

\begin{rem}
	Let $K$ be a degree $d$ extension of $F$ which splits $D$. Assume there is a unique place $\tinf$ 
	of $K$ over $\infty=1/T$. Let $H_K$ be the maximal unramified abelian extension of $K$ in which 
	$\tinf$ splits completely (i.e., $H_K$ is the Hilbert class field of $K$). The Galois group $\Gal(H_K/K)$ is isomorphic to 
	the class group of the integral closure $B$ of $A$ in $K$. By developing a theory of ``complex multiplication'' for Drinfeld-Stuhler modules, 
	the third author has shown in \cite[Thm. 4.10]{PapRMS} that $X^D(H_K)\neq \emptyset$. In particular, if $B$ is a principal ideal domain, 
	then $X^D(K)\neq \emptyset$. 
\end{rem}
	
In Section \ref{sHP}, we assume that $d=2$, so  $X^D$ is a curve. The paper \cite{PapikianCrelle2} contains comprehensive results about 
the existence of rational points on $X^D$ over finite extensions of completions of $F$. 
We combine these results with the results of Section \ref{sGlobPoints} to construct explicit examples of pairs $(X^D, K)$, such that 
\begin{enumerate}
\item[(i)] $[K:F]=2$, 
\item[(ii)] $X^D(K)=\emptyset$, 
\item[(iii)] $X^D(K_v)\neq \emptyset$ for all places $v$ of $K$, 
\end{enumerate}
where $K_v$ denotes the completion of $K$ at $v$. In other words, $X^D$ violates the Hasse principle over $K$. 
Part of the calculations required for checking that the conditions of \cite{PapikianCrelle2}, as well as the conditions of 
Theorem \ref{thmGlobalPoints} or Theorem \ref{thmMain2}, are satisfied for a given pair $(X^D, K)$ were 
performed on a computer using the program \texttt{Magma}. Two such examples 
are the following:
\begin{example} Let $q=3$ and $D$ be the quaternion algebra over $F$ ramified at two primes $\{\fp, \fq\}$ of $A$, 
	and unramified at all other places of $F$. Let $K=F(\sqrt{\fd})$. The Hasse principle is violated for the following choices
	\begin{itemize}
		\item[(a)] $\fp=T^2+T+2, \quad \fq=T^2+1, \quad \fd=T\fp\fq$. 
		\item[(b)] $\fp=T^6+2T^4+T^2+2T+2, \quad \fq=T^2+T+2, \quad \fd=T^{13}+2T+1$. 
	\end{itemize}	
\end{example}

\begin{rem} We should mention the following relevant result: 
	
	Assume $D$ is a quaternion division algebra whose discriminant has degree $\geq 20$. 
	Then there are infinitely many quadratic extensions $K/F$ such that $X^D$ violates the Hasse principle over $K$. 
		
	This is proved in \cite{PapikianCrelle2} by adapting a method of Clark \cite{Clark} to the function field setting, 
	but this method does not give an effective procedure for finding the field extensions $K/F$ over which $X^D$ violates the Hasse 
	principle. 
\end{rem}


\section{Preliminaries}\label{sPreliminaries} We start by fixing the notation that will be used throughout the paper. 
As in the introduction, $\F_q$ will denote a finite field with $q$ elements, $A=\F_q[T]$ the polynomial ring, and $F=\F_q(T)$ the field of fractions of $A$. 
For a nonzero ideal $\fn\lhd A$, by abuse of notation, we denote by the same symbol the 
unique monic polynomial in $A$ generating $\fn$. (It will always be specified or clear from the context whether $\fn$ 
denotes the ideal or its monic generator.) We define 
\begin{equation}\label{eqNormN}
|\fn|=\# (A/\fn). 
\end{equation}

We will call  a nonzero prime ideal of $A$ simply a \textit{prime} of $A$. 
Given a prime $\fp$ of $A$, we denote by $A_\fp$ (resp. $F_\fp$) the completion of $A$ at $\fp$ (resp. the field of fractions of $A_\fp$).  
We denote by $\F_\fp=A/\fp$ the residue field of $A_\fp$, and by $\F_\fp^{(m)}$ the extension of $\F_\fp$ of degree $m\geq 1$.   

The degree $\deg(a)$ of $0\neq a\in A$ is its degree as a polynomial in $T$. The degree function extends to a valuation of $F$. 
The corresponding place of $F$ 
is called the \textit{place at infinity}, and denoted by $\infty$. The normalized absolute value on $F$ at $\infty$ is given by 
$$
|a|=q^{\deg(a)}\quad \text{for }0\neq a\in A. 
$$ 
Note that the normalized absolute value at $\infty$ is closely related with \eqref{eqNormN} via 
$|a|=|(a)|$, where $(a)$ denotes the ideal of $A$ generated by $a$. Since $1/T$ is a uniformizer at $\infty$,  
the completion $\Fi$ of $F$ at $\infty$ is isomorphic to $\F_q(\!(1/T)\!)$. We identify the places of $F$ not equivalent to $\infty$ with 
the primes of $A$. 

Let $D$ be a central simple algebra over $F$ of dimension $d^2$ such that $D\otimes\Fi\cong M_d(\Fi)$. 
Let $\Ram(D)$ be the set of primes of $A$ which ramify in $D$, i.e., $\fp\in \Ram(D)$ 
if and only if $D_\fp:=D\otimes_F F_\fp$ is not isomorphic to $M_d(F_\fp)$. 
Fix a maximal $A$-order $O_D$ in $D$; see \cite{Reiner} for the definitions. Note that $A$ is the center of $O_D$. 
Because $A$ is a principal ideal domain and $D$ is split at $\infty$, any two maximal $A$-orders are conjugate in $D$; see \cite[$\S$34]{Reiner}. 

Given a field $K$ we denote by $\overline{K}$ (resp. $K^\sep$) its algebraic (resp. separable) closure, and put 
$\rG_K=\Gal(K^\sep/K)$. 

Let $K$ be a field equipped with an $A$-algebra structure $\gamma:A\to K$. 
The $A$-characteristic of $K$ is $\chr_A(K):=\ker(\gamma)\lhd A$. We 
will always implicitly consider $F$, and its extensions, as $A$-fields via the natural embedding of $A$ into its field of fractions. 

Let $K[\tau]$ be the skew polynomial ring  with the commutation relation $\tau \alpha=\alpha^q\tau$, $\alpha\in K$. 
One can write the elements of $M_d(K[\tau])$ as finite sums $\sum_{i\geq 0} B_i\tau^i$, where $B_i\in M_d(K)$. 
Using this, we define a homomorphism 
\begin{equation}\label{eq-partial}
\partial: M_d(K[\tau])\To M_d(K), \quad \sum_{i\geq 0} B_i\tau^i\longmapsto B_0.
\end{equation}

A \textit{Drinfeld-Stuhler $O_D$-module} defined over $K$ is an embedding 
	\begin{align*}
	\phi: O_D &\To \End_{\F_q}(\gm_{a, K}^d)\cong M_d(K[\tau]) \\ 
	b &\longmapsto \phi_b
	\end{align*}
	satisfying the following conditions: 
	\begin{itemize}
		\item[(i)] For any $b\in O_D\cap D^\times$, the kernel $\phi[b]$ of the endomorphism $\phi_b$ of $\gm_{a, K}^d$ 
		is a finite group scheme over $K$ of order $\# (O_D/O_D\cdot b)$. 
		\item[(ii)] The composition 
		$$
		A\To O_D\overset{\phi}{\To} M_d(K[\tau]) \overset{\partial}{\To} M_d(K)
		$$
		maps $a\in A$ to the scalar matrix $\gamma(a)I_d$, where $I_d$ denotes the $d\times d$ identity matirx.   
	\end{itemize}

	A \textit{morphism} $u:\phi\to \psi$ between two Drinfeld-Stuhler $O_D$-modules over $K$ is 
	$u\in M_d(K[\tau])$ such that $u\phi_b=\psi_b u$ for all $b\in O_D$. 
	A morphism $u$ is an \textit{isomorphism} if $u$ is invertible in the ring $M_d(K[\tau])$. 
	The set of morphisms $\phi\to \psi$ over $K$ is an $A$-module $\Hom_K(\phi, \psi)$, where 
	$A$ acts by $a\circ u:=u\phi_a$. It can be shown that the kernel of any nonzero morphism $u$ 
	is a finite group scheme over $K$ (i.e., a nonzero morphism is an isogeny), and $\Hom_K(\phi, \psi)$ is a free $A$-module 
	of rank $\leq d^2$; cf. \cite{PapRMS}. 
	We denote $\End_K(\phi)=\Hom_K(\phi, \phi)$ and $\Aut_K(\phi)=\End_K(\phi)^\times$.

Let $\fp$ be a prime of $A$. The $\fp$-adic Tate module of $\phi$ is 
$$
T_\fp(\phi)=\underset{\substack{\longleftarrow \\ n}}{\lim}\ \phi[\fp^n](K^\sep). 
$$
$T_\fp(\phi)$ is a free $A_\fp$-module of rank $\leq d^2$; cf. \cite{Anderson}. 
Moreover, if $\fp\neq \chr_A(K)$, from \cite[Lem. 2.10]{PapRMS}, one deduces an isomorphism 
\begin{equation}\label{eqTfpphi}
T_\fp(\phi)\cong O_D\otimes_A A_\fp
\end{equation} 
of left $O_D$-modules.

	\begin{example}\label{exampleMdA} In the special case when $D=M_d(F)$ and $O_D=M_d(A)$, 
		Drinfeld-Stuhler modules can be obtained from Drinfeld modules by the following construction.  	
		Let $$\Phi: A\To K[\tau], \quad a\longmapsto \Phi_a,$$ be a Drinfeld $A$-module over $L$ of rank $d$. Such a module is 
		uniquely determined by the image of $T$:
		$$
		\Phi_T=\gamma(T)+g_1\tau+\cdots+g_d\tau^d, \quad g_d\neq 0. 
		$$
		Define 
		\begin{align*}
		\phi: O_D &\To M_d(K[\tau]) \\ 
		(a_{ij}) &\longmapsto \left(\Phi_{a_{ij}}\right).
		\end{align*}
		It is easy to check that $\phi$ is a Drinfeld-Stuhler module. 
		In fact, every Drinfeld-Stuhler $M_d(A)$-module arises from some 
		Drinfeld module $\Phi$ via this construction. This is a consequence of the Morita equivalence for Drinfeld-Stuhler modules; see  \cite[$\S$2.4]{PapRMS} 
		for the details.  
	\end{example}

The category of Drinfeld-Stuhler modules over $K$ is equivalent to the category of $\sD$-elliptic sheaves over $K$ 
(modulo a certain action of $\Z$ on the latter category); cf. \cite[$\S$3]{PapRMS}. 
 In \cite{LRS}, $\sD$-elliptic sheaves are defined over any $\F_q$-scheme $S$. The functor 
which associates to $S$ the set of isomorphism classes of $\sD$-elliptic sheaves over $S$ (modulo the action of $\Z$)
possesses a coarse moduli scheme $X^D$ 
over $C:=\p^1_{\F_q}-\Ram(D)-\{\infty\}$ of relative dimension $(d-1)$;   
this follows from \cite[Thm. 4.1]{LRS}, combined with the Keel-Mori theorem. 
Up to isomorphism, $X^D$ does not depend on the choice of a maximal order $O_D$ in $D$. 
Moreover, $X^D$ is geometrically connected since the class number of $A$ is $1$;  
this can be deduced from the rigid-analytic uniformization of $X^D$ described in \cite{BS}.
If $D$ is a central division algebra, then $X^D$ is proper over $C$ by \cite[Thm. 6.1]{LRS}. 
We call $X^D$ the \textit{Drinfeld-Stuhler variety}. 

Assume $\chr_A(K)\not\in \Ram(D)$, so that $\gamma: A\to K$ corresponds to a morphism $\Spec(K)\to C$.  
Let $X^D_K:=X^D\times_C \Spec(K)$.  A Drinfeld-Stuhler module defined over $K$ 
corresponds to a $K$-rational point on $X^D_K$. On the other hand, because $X^D$ is only a coarse moduli scheme, 
it is not necessarily true that every 
$K$-rational point on $X^D_K$ corresponds to some Drinfeld-Stuhler module defined over $K$, 
although this is true if $K$ is algebraically closed (for more on this, see Section \ref{sGlobPoints}).We will denote the set of $K$-rational points on $X^D_K$ 
by $X^D(K)$.


\section{Potentially good reduction property}\label{sPGR}

Let $K$ be a local field of positive characteristic. Let $R$ be the ring of integers of $K$, $\pi$ a uniformizer of $K$, and 
$k=R/(\pi)$ the residue field. Assume $\gamma: A\to R$ is an injective homomorphism. Extending this homomorphism to $\gamma: A\to K$, 
we consider $K$ as an $A$-field  with $\chr_A(K)=0$. 
\begin{defn}
Let $\phi: O_D\to M_d(K[\tau])$ be a Drinfeld-Stuhler module over $K$. We say that $\phi$ has 
\textit{good reduction} if there is $u\in \GL_d(K[\tau])$ such that $\psi=u\phi u^{-1}$  has the following two properties: 
\begin{enumerate}
	\item $\psi: O_D\to M_d(R[\tau])$, i.e., the image of $\psi$ is in $M_d(R[\tau])$; 
	\item $\overline{\psi}: O_D\xrightarrow{\psi}M_d(R[\tau])\xrightarrow{\mod \pi} M_d(k[\tau])$ 
is a Drinfeld-Stuhler module over $k$. 
\end{enumerate}
We say that $\phi$ has \textit{potentially good reduction} 
if there is a finite extension $L/K$ such that $\phi$, considered as a Drinfeld-Stuhler module over $L$, 
has good reduction.
\end{defn}

The main result of this section is the following analogue of a well-known fact about abelian surfaces with 
quaternionic multiplication. 

\begin{thm}\label{thm:main} 
	If $D$ is a central division algebra, then a Drinfeld-Stuhler $O_D$-module $\phi$ over $K$ has potentially good reduction. 
\end{thm}

In principle, this is implicitly proven in \cite[$\S$6]{LRS}  in terms of $\sD$-elliptic sheaves using a result of Drinfeld. 
We will give a different proof in terms of Drinfeld-Stuhler modules, which also provides more information 
about the extension $L/K$ where $\phi$ acquires good reduction. 

To motivate our approach, we recall the proof of potentially good reduction property of abelian surfaces equipped 
with an action of a maximal order in an indefinite division quaternion algebra $D$. 
Let $X$ be such a surface defined over a local field $K$. 
The potentially stable reduction theorem for abelian varieties 
implies that there is a finite extension $L/K$ such that $X_L$ has stable reduction, i.e., the closed fibre 
of the N\'eron model of $X_L$ over the ring of integers of $L$ is an extension of an abelian variety by an algebraic torus $\cT$. 
Moreover, $X_L$ has good reduction if and only if $\cT$ is trivial. To show that $\cT$ is trivial, one considers the 
$\Q$-vector space $V(\cT)=\Hom(\cT\otimes \bar{k}, \gm_{m, \bar{k}})\otimes \Q$. The universal property of N\'eron 
models implies that $D$ acts on $V(\cT)$. Since $D$ is a division quaternion algebra, it cannot act on a nonzero $\Q$-vector space of 
dimension $<4$. Since $\dim V(\cT)\leq 2$, one deduces that $\cT$ is trivial. 

\begin{proof}[Proof of Theorem \ref{thm:main}] 
We will imitate the previous proof in the case of Drinfeld-Stuhler modules. But 
it is not possible to do this directly due to several problems. The most serious of these problems 
is the fact that the potentially stable reduction theorem is simply false for Anderson $A$-modules; cf. \cite{GardeynJNT02}. 
(Drinfeld-Stuhler modules are a special case of Anderson $A$-modules, similar to abelian surfaces with quaternionic multiplication 
being a special case of abelian varieties.) More precisely, in \cite[p. 470]{GardeynJNT02}, Gardeyn gave an explicit 
example of an Anderson $T$-module over $\ls{\F_q}{T}$ whose Tate modules  
do not potentially contain nonzero unramified submodules. 
Another problem is that an analogue of the theory of N\'eron models for Anderson $A$-modules  
is not as robust as for abelian varieties; cf. \cite{GardeynJNT03}.  

First, we adapt an idea used by Taelman in \cite{TaelmanPhD}. 
This will allow us to apply the potentially stable reduction theorem for Drinfeld modules to our problem.  
	
	Recall the following fact. Let $F'/F$ be a finite extension. 
	Then $D\otimes_F F'$ is a central simple algebra over $F'$ such that for a place $w$ of $F'$ 
	over a place $v$ of $F$ we have (cf. \cite[Lem. A.3.2]{LaumonCDV})
	\begin{equation}\label{eqChnageofinv}
	\inv_w(D\otimes_F F')=[F'_w:F_v]\cdot \inv_v(D) \in \Q/\Z. 
	\end{equation}
	Now suppose $F'=\F_{q^n}F$ is obtained by extending the constants. 
	In this case, $F'_w/F_v$ is unramified of degree $n/\gcd(n, \deg(v))$. 
	Hence, using the above formula for the invariants of $D\otimes_F F'$, we see that there is $n$, e.g., 
	$n=d\prod_{v\in \Ram(D)} \deg(v)$, such that the invariants of $D\otimes_F F'$ at all places of $F'$ are $0$, 
	which is equivalent to $D\otimes_F F' \cong M_d(F')$. This fact is known as \textit{Tsen's theorem}.  
	
	Now let $\phi$ be a Drinfeld-Stuhler $O_D$-module over $K$,  and let $n$ be such that $F'=\F_{q^n}F$ splits $D$. 
	Let $A'=\F_{q^n}[T]$ be the integral closure of $A$ 
	in $F'$. We can assume, after possibly extending $K$, that $\F_{q^n}\subset K$. 
	Denote $\sigma=\tau^n$, and consider the composition 
	$$
	\phi':O_D\overset{\phi}{\To} M_d(K[\tau])\xrightarrow{\tau\mapsto \sigma} M_d(K[\sigma]),   
	$$
	where $K[\sigma]$ is the twisted polynomial ring with commutation rule $\sigma b=b^{q^n}\sigma$, $b\in K$. 
	(The second map is a formal substitution $\tau\mapsto \sigma$; it is not a homomorphism.) 
	Note that $\phi'$ 
	is not a Drinfeld-Stuhler module according to our definition, but the definition can be easily generalized so that $\phi'$ is 
	a Drinfeld-Stuhler module of ``rank $n$''. Denote $O_{D'}:=O_D\otimes_{\F_q}\F_{q^n}$, and extend $\phi'$ to an embedding 
	$$
	\widetilde{\phi}':O_{D'} \To M_d(K[\sigma])
	$$
	by mapping $1\otimes\alpha\mapsto \diag(\alpha, \dots, \alpha)$. It is easy to check that $O_{D'}$ 
	is a maximal $A'$-order in $D':=D\otimes_F F'\cong M_d(F')$, for example, by calculating its discriminant.  
	Moreover,  $\widetilde{\phi}'$ 
	is a Drinfeld-Stuhler $O_{D'}$-module in the sense of Section \ref{sPreliminaries}. 
	
	It follows from \cite[(35.14)]{Reiner} that the number of conjugacy classes of maximal $A'$-orders 
	in $D'$ is not larger than the class number of $A'$. Since the class number of $A'$ is $1$, $O_{D'}$ is conjugate to $M_d(A')$, and  
	we can assume $O_{D'}=M_d(A')$. Now the Morita equivalence for Drinfeld-Stuhler modules \cite[Thm. 2.20]{PapRMS} 
	implies that $\widetilde{\phi}'$ arises from a unique Drinfeld $A'$-module $\Phi$ of rank $d$ by the construction 
	of Example \ref{exampleMdA}, i.e., $\widetilde{\phi}'((a_{ij}))=(\Phi_{a_{ij}})$.  
	It is easy to see that any Drinfeld $A'$-module over $K$ 
	acquires stable reduction over a tamely ramified extension $L/K$; cf. \cite{Takahashi}. (Recall that a Drinfeld module $\Phi: A'\to K[\sigma]$ 
	has stable reduction if $\Phi$ is isomorphic over $K$ to a Drinfeld module over $R$ whose reduction modulo $\pi$ is also a Drinfeld module 
	of possibly smaller rank.) 
	After possibly passing to such extension, we assume that $\Phi$ has stable reduction over $K$. Suppose $u\Phi u^{-1}$, $u\in K$, 
	has integral coefficients and stable reduction. Denote $U=\diag(u, \dots, u)$.  Then 
	$$ \widetilde{\psi}':=U\cdot \widetilde{\phi}'\cdot U^{-1}$$ 
	embeds $O_{D'}$ into $M_d(R[\sigma])$. Moreover, for any prime $\fp'$ of $A'$ not equal to $\ker(A'\to k)$  and $m\geq 1$, 
	the kernel of $\widetilde{\psi}'_{(\fp')^m} \Mod{\pi}$ is finite and nonzero. 
	Substituting $\sigma\mapsto \tau$, we see that 
	$\psi:=U \phi U^{-1}$ also has these properties, i.e., $\psi: O_D\to M_d(R[\tau])$ is a Drinfeld-Stuhler module, and the 
	kernel of $\overline{\psi}_{\fp^m}$ on $\gm_{a, k}^d$ is finite and nonzero for any $\fp\neq \ker(A\to k)$ and $m\geq 1$. 
	Therefore, the Tate module $T_\fp(\psi)$ contains a nonzero unramified 
	submodule. Let $M(\psi)$ be the $O_D$-motive associated to $\psi$; see Section 3 of \cite{PapRMS}. 
	The $F$-vector space $M(\psi)\otimes_A F$ is a $D^\opp$-vector space of dimension $1$, where $D^\opp$ denotes the opposite algebra of $D$. 
	The fact that $T_\fp(\psi)$ contains a nonzero unramified submodule, implies that $M(\psi)\otimes F$ contains a nonzero 
	unramified vector subspace $W$. On the other hand, the action of $\rG_K=\Gal(K^\sep/K)$ commutes with the action of $D^\opp$ on $M(\psi)\otimes F$, 
	so $W$ is a $D^\opp$-vector space. Since $D^\opp$ is a division algebra, the dimension of $W$ over $F$ 
	must be at least $d^2$. Thus $M(\psi)\otimes_A F$ is unramified. This puts us in a position where we can apply Gardayn's criterion 
	for good reduction of Anderson's motives \cite{GardeynJNT02} to conclude that $\phi$ has good reduction. 
\end{proof}

\begin{prop}\label{lem4.5} Assume $D$ is a central division algebra and $\phi$ is a 
	Drinfeld-Stuhler $O_D$-module over $K$. 
	There is a totally tamely ramified extension $L/K$ of degree dividing $q^d-1$ over which $\phi$ has good reduction.  
\end{prop}
\begin{proof}
	First we prove that there is a tamely ramified extension over which $\phi$ has good reduction. 
	After possibly extending $\F_q\subset K$ to a large enough $\F_{q^n}$, we are in the set up of the 
	the proof of Theorem \ref{thm:main}.  It follows from the proof that the Drinfeld module $\Phi: A'\to K[\sigma]$ 
	has potentially good reduction. 
	It is not hard to show that a Drinfeld module with potentially good reduction 
	acquires good reduction over a tamely ramified extension of $K$; see \cite[Prop. 2]{Takahashi}. But then, again 
	from the proof of Theorem \ref{thm:main}, it follows that 
	$\phi$ also acquires good reduction over a tamely ramified extension of $K$. 
	
	Now we prove that there is a totally tamely ramified extension over which $\phi$ has good reduction. 
	The fact that $\phi$ has potentially good reduction implies that the image 
	of the inertia group $I_K$ under the representation $\rho: \rG_K\to \Aut_{O_D}(T_\fp(\phi))$ is a finite group for any prime $\fp$ 
	not equal to the $A$-characteristic of $k$. 
	The kernel $N$ of $\rho$ restricted to $I_K$ is independent of $\fp\neq \chr_A(k)$. 
	Let $\theta\in \rG_K$ be a Frobenius element. Denote by $\G_\theta$ the closure 
	of the subgroup of $\rG_K$ generated by $\theta$. 
	Let $L$ be the extension of $K$ cut out by $\G_\theta N\subset \rG_K$. Then $L/K$ 
	is a totally tamely ramified extension. Since the action of the inertia group $I_L$ on $T_\fp(\phi)$ is trivial, 
	by Gardeyn's good reduction criterion of $T$-motives \cite{GardeynJNT02}, $\phi$ has good reduction over $L$. 
	
	Finally, we prove the claim about the degree of $L/K$. We have proved that $\phi$ has good reduction 
	over an abelian extension $L$ of $K$. We can assume that $\phi$ does not have good reduction over any proper 
	subfield of $L$. Fix  a prime $\fp\lhd A$ such that $\fp\neq \chr_A(k)$. Any $g\in \Gal(L/K)$, $g\neq 1$, acts non-trivially on $T_\fp(\phi)$, since 
	otherwise the fixed field $K'$ of the subgroup generated by $g$ is an extension 
	over which $T_\fp(\phi)$ is unramified, and thus  
	by Gardayn's criterion $\phi$ has good reduction over $K'$. 
	There is a natural isomorphism $T_\fp(\phi)\cong T_\fp(\bar{\phi})$ given by reduction 
	modulo the maximal ideal of the ring of integers of $L$.  Hence $g\in \Aut_{\rG_k}(T_\fp(\bar{\phi}))$. On the other hand, 
	by the analogue of Tate's isogeny conjecture for $T$-modules proved by Taguchi and Tamagawa \cite{Taguchi}, we have 
	an isomorphism 
	$$
	\iota: \End_{k, T}(\bar{\phi})\overset{\sim}{\To} \End_{\rG_k}(T_\fp(\bar{\phi})), 
	$$
	where on the left hand side we consider $\bar{\phi}$ as a $T$-module and  $$\End_{k, T}(\bar{\phi}):=\{u\in M_d(k[\tau])\mid  
	u\bar{\phi}_T=\bar{\phi}_T u\}.$$ 
	Thus, $g$ induces a non-trivial automorphism of $\bar{\phi}$ over $k$. This automorphism commutes with the action of $O_D$, 
	since the action of $O_D$ on $T_\fp(\phi)$ commutes with $\Gal(L/K)$ over $K$. Thus, $\Gal(L/K)\hookrightarrow 
	 \Aut_{k}(\bar{\phi})$, where now 
	$\bar{\phi}$ is considered as a Drinfeld-Stuhler module. 
	Finally, by \cite[Thm. 4.1]{PapRMS}, $\Aut_{k}(\bar{\phi})\cong \F_{q^m}^\times$   
	for some $m\mid d$. 
\end{proof}


\section{Canonical isogeny characters}\label{sCanChar}

Let $\fp\in \Ram(D)$. In this section we assume that $D_\fp:=D\otimes_F F_\fp$ is a division algebra with Hasse 
invariant $\inv_\fp(D)=1/d$. 

We start by examining the reduction of $O_D$ modulo $\fp$. 
Note that 
$$
O_D\otimes_A \F_\fp=(O_D\otimes_A A_\fp)\otimes_{A_\fp} \F_\fp. 
$$
Since $\cD_\fp:=O_D\otimes_A A_\fp$ is a maximal order in the central division algebra $D_\fp$ over $F_\fp$, it can be 
explicitly described as follows (cf. \cite[Appendix A]{LaumonCDV}):   
Let $\F_\fp^{(d)}$ be the degree $d$ extension of $\F_\fp$. Let $F_\fp^{(d)}=\F_\fp^{(d)} F_\fp$ be the unramified extension of $F_\fp$ of degree $d$  
and $A_\fp^{(d)}$ be the ring of integers of $F_\fp^{(d)}$. Let $\sigma\in \Gal(F_\fp^{(d)}/F_\fp)$ 
be the lifting of the Frobenius automorphism $\alpha\mapsto \alpha^{|\fp|}$ in $\Gal(\F_\fp^{(d)}/\F_\fp)$.  By \cite[$\S$14]{Reiner}, 
$$
D_\fp\cong F_\fp^{(d)}[\Pi]/(\Pi^d-\fp),
$$
where $F_\fp^{(d)}[\Pi]$ is the non-commutative polynomial ring in $\Pi$ over $F_\fp^{(d)}$ 
with commutation rule 
$$
\Pi \alpha=\sigma(\alpha)\Pi, \qquad \alpha\in F_\fp^{(d)}. 
$$
The maximal order of $D_\fp$ is 
$$
\cD_\fp=A_\fp^{(d)}[\Pi]/(\Pi^d-\fp)
$$ 
and $\Pi \cD_\fp=\cD_\fp\Pi$ 
is the maximal ideal of $\cD_\fp$. To make this description even more explicit, note that 
$A_\fp^{(d)}$ may be identified with the ring $\dvr{\F_\fp^{(d)}}{\fp}$ of formal series with coefficients in $\F_\fp^{(d)}$, where 
we consider $\fp$ as a uniformizer of $F_\fp$;  cf. \cite[$\S$II.4]{SerreLF}.  
From this description we obtain 
\begin{align}\label{eqcD_fp}
O_D/\fp O_D & \cong \cD_\fp/\fp\cD_\fp  \\ 
 \nonumber  & \cong  \F_\fp^{(d)}[\Pi]/\Pi^d  \\ 
 \nonumber &\cong   \F_\fp^{(d)}\oplus \F_\fp^{(d)}\Pi\oplus \cdots\oplus \F_\fp^{(d)}\Pi^{d-1},\quad \Pi^d=0, \quad \Pi\alpha=\alpha^{|\fp|}\Pi. 
\end{align}

Now let $K$ be an $A$-field $\gamma: A\to K$ such that $\chr_A(K)\neq \fp$. Let $\phi$ be a Drinfeld-Stuhler $O_D$-module over $K$. 
By \cite[Lem. 2.10]{PapRMS}, we have an isomorphism of left $O_D$-modules 
\begin{equation}\label{eqphi[p]}
\phi[\fp]\cong O_D/\fp O_D. 
\end{equation}

\begin{defn}
It is easy to see from \eqref{eqcD_fp} and \eqref{eqphi[p]} that $\phi[\fp]$ has exactly one 
$O_D$-submodule which is a $1$-dimensional $\F_\fp^{(d)}$-vector space, namely the $\F_\fp^{(d)}$-vector space spanned by $\Pi^{d-1}$. 
We shall denote this submodule by $\cC_{\phi, \fp}$ and call it 
the \textit{canonical subgroup} of $\phi[\fp]$. Note that the canonical subgroup can be equivalently described as the 
kernel of $\Pi$ acting on $\phi[\fp]$. Since the action of $\rG_K$ on $\phi[\fp]$ commutes with 
the action of $O_D$, the canonical subgroup $\cC_{\phi, \fp}$ is rational over $K$. 
\end{defn}

\begin{prop}\label{lem:Sec2lem2.5}
	Assume $K$ is a local field with uniformizer $\varpi$, and $\gamma(\fp)=\varpi^m$, $m\geq 1$. Further assume that 
	$K$ contains $\F_\fp^{(d)}$ and 	
	$\phi$ has good reduction over $K$. Then the extension $K(\cC_{\phi, \fp})/K$ obtained from the canonical subgroup of $\phi[\fp]$
	is a totally tamely ramified extension of degree $(|\fp|^d-1)/\gcd(|\fp|^d-1, m)$. 
\end{prop}
\begin{proof}
	Let $R$ be the ring of integers of $K$. 
	Because $\phi$ has good reduction over $R$, each group-scheme $\phi[\fp^n]$, $n\geq 1$, over $K$
	extends to a finite flat group-scheme over $R$ of order $|\fp|^{d^2n}$. On the other hand, by \cite[Lem. 5.6]{PapRMS}, 
	the closed fibre of $\phi[\fp^n]$ over $R$ is connected. (Here we use the assumption that $D_\fp$ is a division algebra.)
	We can form the $\fp$-divisible group $\G(\phi)=\underset{\To}{\lim}\ \phi[\fp^n]$ in a usual manner; cf. \cite[p. 33]{LaumonCDV}. 
	Note that each $\phi[\fp^n]$ can be embedded as a subgroup-scheme into $\gm_{a, R}^d$, and $\cD_\fp$ acts on $\G(\phi)$ 
	since it acts on each $\phi[\fp^n]$ 
	via its quotients $\cD_\fp/\fp^n=O_D/\fp^n$ in a compatible manner. Let $\widehat{\G}(\phi)$ be the formal group associated 
	to $\G(\phi)$ by the Serre-Tate construction \cite[Prop. 1]{TatePDiv}. As a formal group, $\widehat{\G}(\phi)$ is just 
	the direct product of $d$ copies of the formal additive group $\widehat{\gm}_{a, R}$, but it comes equipped with an action of $\cD_\fp$. 
	From this action we get an embedding  
	$$
	\widehat{\phi}: \cD_\fp\to \End_{\F_q}(\widehat{\gm}_{a, R}^d)\cong M_d\left(\dvr{R}{\tau}\right),
	$$ 
	where $\dvr{R}{\tau}$ is the non-commutative ring of formal series in $\tau$ with coefficients in $R$ and the 
	commutation rule $\tau a=a^q\tau$, $a\in R$. 
	Note that the composition 
	$$
	A_\fp\to \cD_\fp\xrightarrow{\widehat{\phi}} M_d\left(\dvr{R}{\tau}\right) \xrightarrow{\partial} M_d(R),
	$$
	where $\partial$ is defined similarly to \eqref{eq-partial}, 
	maps $a\in A_\fp$ to $\diag(\gamma(a), \dots, \gamma(a))$ and $\phi[\fp^n]=\widehat{\phi}[\fp^n]$ for all $n\geq 1$. 
	
	As earlier, represent $\cD_\fp\cong \dvr{\F_\fp^{(d)}}{\fp}[\Pi]/(\Pi^d-\fp)$ and consider the action of $\F_\fp^{(d)}$ on the 
	tangent space $\Lie(\widehat{\gm}_{a, K}^d)\cong K^{\oplus d}$. Since $K$ contains a subfield isomorphic to $\F_\fp^{(d)}$, there 
	is at least one eigenspace $V\subset K^{\oplus d}$ on which $\F_\fp^{(d)}$ acts by scalar multiplication via an embedding $\F_\fp^{(d)}\hookrightarrow K$ 
	(compatible with the $\F_\fp$-structure of both fields), i.e., by a ``fundamental character''. If we fix one such character $\chi$, then 
	any other fundamental character is of the form $\chi_i:=\chi^{|\fp|^i}$, $0\leq i\leq d-1$. On the other hand, the action of $\F_\fp^{(d)}$ 
	on $\Lie(\widehat{\gm}_{a, K}^d)$ is compatible with the action of $\Pi$. Let $V_{\chi_i}\subset \Lie(\widehat{\gm}_{a, K}^d)$ be the eigenspace on which $\F_\fp^{(d)}$ acts via $\chi_i$.  Since $\Pi\chi_i=\chi_i^{|\fp|}\Pi$ and $\Pi^d=\fp$, we see that $\Lie(\widehat{\gm}_{a, K}^d)$ decomposes into a 
	direct sum $\bigoplus_{i=0}^{d-1}V_{\chi_i}$. By comparing the dimensions, one concludes that each $V_{\chi_i}$ is $1$-dimensional. 
	
	We assume without loss of generality that 
	$$\Lie(\widehat{\gm}_{a, K}^d)=\bigoplus_{i=0}^{d-1} \Lie(\widehat{\gm}_{a, K})=\bigoplus_{i=0}^{d-1} V_{\chi_i}, 
	$$ 
	i.e., the decomposition of the formal group into a direct product of $\widehat{\gm}_{a, K}$'s is compatible 
	with the decomposition of the tangent space into $\F_\fp^{(d)}$-eigenspaces. This decomposition of the formal group 
	is preserved by the action of $\widehat{\phi}(A_\fp)$, since $A_\fp$ commutes with $\F_\fp^{(d)}$ in $\cD_\fp$. 
	Hence from each component we get a formal $A_\fp$-module 
	$$
	\widehat{\psi}_i: A_\fp \to \End_{\F_q}(\widehat{\gm}_{a, R})=\dvr{R}{\tau}, 
	$$ 
	such that $(\widehat{\psi}_i)_\fp\equiv \gamma(\fp) \Mod{\tau}$, where  $(\widehat{\psi}_i)_\fp$ is the image of $\fp$ 
	under $\widehat{\psi}_i$. 
	Moreover, $\phi[\fp]=\widehat{\phi}[\fp]=\bigoplus_{i=0}^{d-1}\widehat{\psi}_i[\fp]$. From the construction, we have 
	$\Pi \widehat{\psi}_i[\fp]= \widehat{\psi}_{i+1}[\fp]$ for $i=0, \dots, d-2$ and $\Pi \widehat{\psi}_{d-1}[\fp]=0$. 
	Hence $\cC_{\phi, \fp}=\widehat{\psi}_{d-1}[\fp]$.  
	
	It remains to examine $\widehat{\psi}:=\widehat{\psi}_{d-1}$ more closely. Since $A_\fp$ commutes with $\F_\fp^{(d)}$, we in fact have 
	$$
	\widehat{\psi}: A_\fp \to \End_{\F_\fp^{(d)}}(\widehat{\gm}_{a, R})=R[\![\tau_\fp^{d}]\!], 
	$$
	where $\tau_\fp=\tau^{\deg(\fp)}$. 
	The formal group $\widehat{\phi}$ modulo ${\varpi}$  
	has height ${d^2}$, so $\widehat{\psi}$ modulo ${\varpi}$ has height $d$, i.e., after possibly replacing $\widehat{\psi}$ by 
	an isomorphic module, we have 
	$$
	\widehat{\psi}_\fp -\tau_\fp^{d} \in \varpi R[\![\tau_\fp^{d}]\!]. 
	$$
	By assumption $\gamma(\fp)=\varpi^m$, so the power series $\widehat{\psi}_\fp(x)\in R[\![x]\!]$ satisfies 
	\begin{itemize}
		\item $\widehat{\psi}_\fp(x)=\varpi^m x + \text{terms of degree $\geq 2$}$; 
		\item $\widehat{\psi}_\fp(x) \equiv x^{|\fp|^d}\Mod{\varpi}$. 
	\end{itemize}
	Thus, $\widehat{\psi}$ is a Lubin-Tate formal group. This implies, by local class field theory, that $K(\cC_{\phi, \fp})/K$ is totally ramified of degree $(|\fp|^d-1)/\gcd(|\fp|^d-1, m)$. 
\end{proof}

The action of $\rG_K$ on the subgroup $\cC_{\phi, \fp}$ yields a character 
$$
\varrho_{\phi,\fp}: \rG_K\to \Aut_{O_D}(\cC_{\phi, \fp})\approx (\F_\fp^{(d)})^\times. 
$$ 
This character depends on an identification $\cD_\fp/\Pi\cD_\fp \cong \F_\fp^{(d)}$; such identifications differ by automorphisms 
of $\F_\fp^{(d)}$ given by powers of the Frobenius $\alpha\mapsto \alpha^{|\fp|}$. 
\begin{defn}
The characters $\varrho_{\phi,\fp}^{|\fp|^i}$, $0\leq i\leq d-1$, will be called the \textit{canonical isogeny characters} associated to $\phi$ at $\fp$.  
\end{defn}

For the rest of this section we study the properties of canonical isogeny characters. 
By \eqref{eqcD_fp} and \eqref{eqphi[p]}, $\phi[\fp]$ is a $d$-dimensional vector space over $\F_\fp^{(d)}$, so from the action of $\rG_K$ on $\phi[\fp]$ 
one obtains a representation 
\begin{equation}\label{eqpi(g)def}
\pi_{\phi, \fp}: \rG_K\to \Aut_{O_D}(\phi[\fp])\subset \GL_d(\F_\fp^{(d)}). 
\end{equation}
\begin{lem}\label{lemCCdetpi}  
	$$\det(\pi_{\phi,\fp})=\varrho_{\phi,\fp}^{1+|\fp|+\cdots+|\fp|^{d-1}}=\Nr_{\F_\fp^{(d)}/\F_\fp}\left(\varrho_{\phi,\fp}\right).
	$$
\end{lem}
\begin{proof}
Fix $\{1, \Pi, \dots, \Pi^{d-1}\}$ as a  basis of $O_D/\fp O_D$ considered as an 
$\F_\fp^{(d)}$-vector space. Let $g\in \rG_K$ and 
$$
g\circ 1=\alpha_1+\alpha_2 \Pi+\cdots+\alpha_{d}\Pi^{d-1}, 
$$
where $\alpha_1\in (\F_\fp^{(d)})^\times$ and $\alpha_2, \dots, \alpha_d\in \F_\fp^{(d)}$. 
Since the action of $\Pi$ commutes with the action of $g$, we get 
$$
g\circ \Pi =\Pi(g\circ 1)=\alpha_1^{|\fp|}\Pi+\alpha_2^{|\fp|} \Pi^2+\cdots+\alpha_{d-1}^{|\fp|}\Pi^{d-1}.
$$
Continuing in this manner, we obtain
\begin{equation}\label{eqpi(g)}
\pi_{\phi,\fp}(g)=\begin{pmatrix} \alpha_1 & 0 & 0 &  \cdots &  0 \\ 
\alpha_2 & \alpha_1^{|\fp|} & 0 &   \cdots &  0 \\ 
\alpha_3 & \alpha_2^{|\fp|} & \alpha_1^{|\fp|^{2}} &  \cdots &  0 \\ 
\vdots & \vdots & \vdots &  \ddots &  \vdots\\
\alpha_d & \alpha_{d-1}^{|\fp|} & \alpha_{d-2}^{|\fp|^2} &  \cdots &  \alpha_1^{|\fp|^{d-1}}
\end{pmatrix}.
\end{equation}
The canonical isogeny characters are the diagonal entries of $\pi_{\phi,\fp}$. Now the claim of the lemma is clear. 
\end{proof}

Let $M(\phi)$ be the $O_D$-motive associated to $\phi$; cf. \cite[$\S$3]{PapRMS}. Recall that $M(\phi)$ is a left $O_D^\opp\otimes_{\F_q} K[\tau]$-module,  
which is locally free $O_D^\opp\otimes_{\F_q} K$-module of rank $1$ and free $K[\tau]$-module of rank $d$. 
A construction of Lafforgue \cite[p. 26]{Lafforgue} associates to $M(\phi)$ an $A$-motive $\det(M(\phi))$ of rank $1$ and dimension $1$, 
along with a map $\det: M(\phi)\to \det(M(\phi))$ which on $M(\phi)$ as a locally free $O_D^\opp\otimes K$-module of rank $1$ is 
given by the reduced norm on $D^\opp$. Anderson's duality \cite[Thm. 1]{Anderson} associates to $\det(M(\phi))$ a Drinfeld $A$-module of rank $1$, 
which we will call the \textit{determinant} of $\phi$ and denote by $\det(\phi)$. 
Let $$\chi_{\phi, \fp}: \rG_K\to \Aut(\det(\phi)[\fp]) \cong \F_\fp^\times$$ be the character by which $\rG_K$ acts on the $\fp$-torsion of $\det(\phi)$. 

\begin{prop}\label{prop_chiphi} 
	$$
	\chi_{\phi, \fp}= \Nr_{\F_\fp^{(d)}/\F_\fp}(\varrho_{\phi, \fp}). 
	$$
\end{prop}
\begin{proof}
	By \eqref{eqcD_fp}, the map 
	\begin{align*} 
	\iota: O_D/\fp O_D &\to M_d(\F_\fp^{(d)})\\ 
	\alpha_1+\alpha_2\Pi+\cdots+\alpha_d\Pi^{d-1} & \mapsto 
	\begin{pmatrix} \alpha_1 & \alpha_2 & \alpha_3 &  \cdots &  \alpha_d \\ 
	0 & \alpha_1^{|\fp|} & \alpha_2^{|\fp|} &   \cdots &  \alpha_{d-1}^{|\fp|} \\ 
	0 & 0 & \alpha_1^{|\fp|^{2}} &  \cdots &  \alpha_{d-2}^{|\fp|^2} \\ 
	\vdots & \vdots & \vdots &  \ddots &  \vdots\\
	0 &0 &0 &  \cdots &  \alpha_1^{|\fp|^{d-1}}
	\end{pmatrix}
	\end{align*}
	is an embedding. (As an algebra, $O_D/\fp O_D$ is generated by $\F_\fp^{(d)}$ and $\Pi$ and, as one easily checks, 
	$\iota(\Pi)\iota(\alpha)=\iota(\alpha)^{|\fp|}\iota(\Pi)$ for any $\alpha\in \F_\fp^{(d)}$.) Hence the reduced norm $\Nr: O_D\to A$ 
	induces the map 
	\begin{align}\label{eqNrphi[p]}
	\Nr: O_D/\fp O_D &\to \F_\fp, \\ 
	\alpha_1+\alpha_2\Pi+\cdots+\alpha_d\Pi^{d-1} &\mapsto \alpha_1^{1+|\fp|+\cdots+|\fp|^{d-1}}=\Nr_{\F_\fp^{(d)}/\F_\fp}(\alpha_1). \nonumber
	\end{align} 
	
	By \cite[Prop. 1.8.3]{Anderson}, $\phi[\fp]$ is dual to $M(\phi)/\fp M(\phi)$, hence the reduced norm 
	$$M(\phi)/\fp M(\phi) \to \det(M(\phi))/\fp \det(M(\phi))
	$$ 
	corresponds to the map 
	$$\Nr: \phi[\fp]\cong O_D/\fp O_D\to \F_\fp\cong \det(\phi)[\fp]$$ constructed in \eqref{eqNrphi[p]}. 
	Comparing this with Lemma \ref{lemCCdetpi}, we see that the following diagram commutes 
	$$
	\xymatrix{\phi[\fp] \ar[rrr]^{\pi_{\phi, \fp}(g)} \ar[d]_{\Nr} & & &\phi[\fp] \ar[d]_{\Nr} \\ 
		\det(\phi)[\fp] \ar[rrr]^{\Nr_{\F_\fp^{(d)}/\F_\fp}(\varrho_{\phi, \fp}(g))} & & &
		\det(\phi)[\fp]}
	$$
	Since Lafforgue's determinant construction is equivariant with respect to the action of $\rG_K$,   
	we conclude that $\chi_{\phi, \fp}(g)=\Nr_{\F_\fp^{(d)}/\F_\fp}(\varrho_{\phi, \fp}(g))$ for all $g\in \rG_K$. 
\end{proof}

\begin{lem}\label{lem_CarChi}
Let $C: A\to K[\tau]$, $C_T=\gamma(T)+\tau$, be the Carlitz module over $K$. Let 
$$\chi_{C, \fp}: \rG_K\to \Aut(C[\fp])\cong \F_\fp^\times. 
$$	
Then, 
$$\chi_{\phi, \fp}^{q-1}=\chi_{C, \fp}^{q-1}.
$$
\end{lem}
\begin{proof}
	Denote $\rho:=\det(\phi)$. Then $\rho$ is defined by $\rho_T=\gamma(T)+b\tau$ for some $0\neq b\in K$. 
	Let $\beta$ be a fixed $(q-1)$-th root of $b$, so that $\beta \rho_T\beta^{-1}=C_T$. Denote by $C_\fp(x)$ 
	the $\F_q$-linear polynomial whose roots constitute $C[\fp]$ (for example, $C_T(x)=\gamma(T)x+x^q$), and similarly for $\rho_\fp(x)$.  Then 
	$\beta \rho_\fp(x)=C_\fp(\beta x)$. This implies that multiplication by $\beta$ gives an isomorphism 
	$$
	\beta: \rho[\fp] \xrightarrow{\sim}  C[\fp] \quad  
	\alpha \mapsto \beta\alpha. 
	$$
	Let 
	$\chi_b: \rG_K \to \F_q^\times$
	be the character $g\mapsto g(\beta)/\beta$, which is independent of the choice of the $(q-1)$-th root $\beta$ of $b$. We see 
	from the above isomorphism that 
	$
	\chi_{C, \fp}=\chi_{\phi, \fp}\otimes \chi_b. 
	$
	Finally, since $\chi_b^{q-1}=1$, we get $\chi_{C, \fp}^{q-1}=\chi_{\phi, \fp}^{q-1}$. 
\end{proof}

\begin{cor}\label{cor2.3}
	$$\Nr_{\F_\fp^{(d)}/\F_\fp}(\varrho_{\phi, \fp})^{q-1}=\chi_{C, \fp}^{q-1}.$$ 
\end{cor}
\begin{proof}
	Follows from Proposition \ref{prop_chiphi} and Lemma \ref{lem_CarChi}. 
\end{proof}

Now suppose that $K$, the field over which the Drinfeld-Stuhler $O_D$-module $\phi$ is defined, is a finite extension of $F$.  
Denote by $K^\ab$ the maximal abelian extension of $K$ in $K^\sep$. Let $\rG_K^\ab=\Gal(K^\ab/K)$. Note that a canonical isogeny claracter 
factors through $\varrho_{\phi, \fp}:\rG_K^\ab\to (\F_\fp^{(d)})^\times$. Given a place $v$ of $K$, denote by $K_v$ (resp. $O_v$) the completion 
of $K$ at $v$ (resp. the ring of integers in $K_v$). Let 
$$
\omega_v: K_v^\times \To \rG_K^\ab
$$
be the \textit{local Artin homomorphism} (mapping $K_v^\times$ to the decomposition group of $v$ in $\rG_K^\ab$).  
Let 
$
\widetilde{r}_{\phi, \fp}(v): K_v^\times \to (\F_\fp^{(d)})^\times$ be the composition 
\begin{equation}\label{eqwidetilder}
K_v^\times \xrightarrow{\omega_v} \rG_K^\ab \xrightarrow{\varrho_{\phi, \fp}} (\F_\fp^{(d)})^\times, 
\end{equation}
and let $r_{\phi, \fp}(v): O_v^\times \to (\F_\fp^{(d)})^\times$ be the restriction of $\widetilde{r}_{\phi, \fp}(v)$ to $O_v^\times$. 

\begin{prop}\label{propUnramPinf} With notation as above, we have:
	\begin{enumerate}
		\item If $v$ does not lie over $\fp$ or $\infty$, then $r_{\phi, \fp}(v)^{q^d-1}=1$. 
		\item If $v$ lies over $\infty$, then $\widetilde{r}_{\phi, \fp}(v)^{\frac{|\fp|^d-1}{|\fp|-1}(q-1)}=1$. 
	\end{enumerate}
\end{prop}
\begin{proof}
By Proposition \ref{lem4.5}, $\phi$ acquires good reduction over a totally tamely ramified extension $L$ of $K_v$ 
of degree dividing $q^d-1$.  If $v$ does not lie over $\fp$ or $\infty$, then $\phi[\fp]$ is unramified over $L$ 
(as $\phi[\fp]$ extends to an \'etale group scheme over the ring of integers of $L$). Hence 
$\varrho_{\phi, \fp}^{q^d-1}$ is the trivial character when restricted to the inertia group at $v$. 
Since by local class field theory $\omega_v$
maps $O_v^\times$ into the inertia group at $v$, we conclude that $r_{\phi, \fp}(v)^{q^d-1}$ is the trivial homomorphism.  
This proves (1). 

To prove (2), first observe that by Corollary \ref{cor2.3} we have 
$$\varrho_{\phi, \fp}^{\frac{|\fp|^d-1}{|\fp|-1}(q-1)}=\chi_{C, \fp}^{q-1}.$$
Hence it is enough to show that $\chi_{C, \fp}^{q-1}$ is trivial when considered as a character of $\rG_{\Fi}$. Let 
$\La_C\subset \Ci$ be the $A$-lattice of rank $1$ corresponding to the Carlitz module $C$, where $\Ci$ is the completion of 
$\overline{F}_\infty$; cf. \cite[$\S$3]{Drinfeld}. 
Carlitz explicitly computed $\La_C$, and from that calculation one easily deduces that 
$\Fi(\La_C)=\Fi(\sqrt[q-1]{T-T^q})$; see \cite[p. 236]{Rosen}. In particular, $[\Fi(\La_C):\Fi]=q-1$. On the 
other hand, any torsion point of $C$ is rational over $\Fi(\La_C)$; cf. \cite[Exercise 13.10]{Rosen}. This implies that 
$\chi_{C, \fp}$ restricted to $\rG_{\Fi}$ has order dividing $q-1$. Hence $\chi_{C, \fp}^{q-1}=1$ on $\rG_{\Fi}$. 
\end{proof}

Fix a prime $\fP$ of $K$ over $\fp$. Let $f_\fP$ be the residue degree of $\fP$ over $\fp$, i.e., the residue 
field $\F_\fP=\F_\fp^{(f_\fP)}$ of $\fP$ is a degree $f_\fP$ extension of $\F_\fp$. Let 
$$
t_\fP=\gcd(f_\fP, d). 
$$
For $u\in O_\fP^\times$ denote by $\bar{u}\in \F_\fP^\times$ the reduction of $u$ modulo $\fP$. 

\begin{lem}\label{lemrfP}
	There is a unique integer $0\leq c_\fP\leq |\fp|^{t_\fP}-1$ such that 
	$$
	r_{\phi, \fp}(\fP)(u)=\Nr_{\F_\fP/\F_\fp^{(t_\fP)}}(\bar{u})^{-c_\fP}
	$$
	for all $u\in O_\fP^\times$. 
\end{lem}
\begin{proof}
	The Artin homomorphism $\omega_\fP$ maps $O_\fP^\times$ into the inertia subgroup $I_\fP\subset \rG_K^\ab$ of $\fP$. 
	Since $K(\cC_{\phi, \fp})/K$ is tamely ramified, $	r_{\phi, \fp}(\fP)$ factors through the tame quotient of $I_\fP$, which is isomorphic to 
	$(O_\fP/\fP)^\times\cong \F_\fP^\times$; 
	cf. \cite[$\S$IV. 2]{SerreLF}. The image of any homomorphism $\F_\fP^\times\to (\F_\fp^{(d)})^\times$ is contained 
	in the unique cyclic subgroup of $(\F_\fp^{(d)})^\times$ of order 
	$$
	\gcd(|\fp|^{f_\fP}-1, |\fp|^d-1)=|\fp|^{\gcd(f_\fP, d)}-1=|\fp|^{t_\fP}-1.
	$$
	Thus, $r_{\phi, \fp}(\fP)$ is a homomorphism $O_\fP^\times\xrightarrow{\Mod{\fP}} \F_\fP^\times\to (\F_\fp^{(t_\fP)})^\times$. 
	Finally, observe that the norm homomorphism $\Nr_{\F_\fP/\F_\fp^{(t_\fP)}}: \F_\fP^\times\to (\F_\fp^{(t_\fP)})^\times$ 
	is surjective, and, since both groups are cyclic, basic group theory implies that there is a unique $0\leq c_\fP\leq |\fp|^{t_\fP}-1$ 
	such that $r_{\phi, \fp}(\fP)(u)=\Nr_{\F_\fP/\F_\fp^{(t_\fP)}}(\bar{u})^{-c_\fP}$ for all $u\in O_\fP^\times$. 
\end{proof}

\begin{lem}\label{lemchi_fP}
	Let $e_\fP$ be the ramification index of $\fP$ over $\fp$. Let $\chi_\fP: O_\fP^\times\to \F_\fp^\times$ be the 
	composition 
	$$
	O_\fP^\times\xrightarrow{\omega_\fP} \rG_K^\ab\xrightarrow{\chi_{C, \fp}} \Aut(C[\fp])\cong \F_\fp^\times. 
	$$
	Then $\chi_\fP(u)=\Nr_{\F_\fP/\F_\fp}(\bar{u})^{-e_\fP}$ for all $u\in O_\fP^\times$. 
\end{lem}
\begin{proof}
	Let $C_\fp(x)=\fp x+\cdots+x^{|\fp|}$ be the linearized polynomial corresponding to $C_\fp$, where 
	$C_\fp$ is the image of $\fp$ under the Carlitz module homomorphism $C: A\to K[\tau]$. By \cite{HayesCFT},  
	$f(x)=C_\fp(x)/x$ is irreducible and separable over $F$, and the splitting field $L$ of $f$ is totally tamely ramified over $\fp$.  
  Let $v$ denote the unique extension of the normalized valuation on $K_\fP$ to $K_\fP^\sep$. For any root $\alpha$ 
  of $f$ we have 
 $
  v(\alpha)=e_\fP/(|\fp|-1). 
  $
  Let $\pi$ be a uniformizer of $K_\fP$. Let $\theta_{|\fp|-1}: I_\fP\to \F_\fp^\times$ 
  be the character 
  $g\mapsto g (\pi^{1/(|\fp|-1)})/\pi^{1/(|\fp|-1)}$ 
  of the inertia subgroup $I_\fP\subseteq \rG_K^\ab$ of $\fP$. 
  This character factors through the tame quotient $I_t$ of $I_\fP$. 
  According to \cite[Prop. 7]{SerreInventiones}, because $v(\alpha)=e_\fP/(|\fp|-1)$,  
  $$g(\alpha)/\alpha=\theta_{|\fp|-1}(g)^{e_\fP}\quad \text{for all } g\in I_t,$$ 
  i.e., $I_t$ acts on the roots of $f(x)$ by the character $\theta_{|\fp|-1}^{e_\fP}$. Thus, 
  \begin{equation}\label{eqLem3.9.1}
  \chi_{C, \fp}=\theta_{|\fp|-1}^{e_\fP}. 
  \end{equation}
  Next, by 
  \cite[$\S$1.3]{SerreInventiones}, for $\bar{u}\in \F_\fP^\times\subset I_t$, we have 
  \begin{equation}\label{eqLem3.9.2}
  \theta_{|\fp|-1}(\bar{u})^{e_\fP}=\Nr_{\F_\fP/\F_\fp}(\bar{u})^{e_\fP}. 
  \end{equation}
  Finally, by \cite[Prop. 3]{SerreInventiones}, for $u\in O_\fP^\times$, we have 
  \begin{equation}\label{eqLem3.9.3}
  \theta_{|\fp|-1}(\omega_\fP(u))=\theta_{|\fp|-1}(\bar{u}^{-1}). 
  \end{equation}
  Combining \eqref{eqLem3.9.1}, \eqref{eqLem3.9.2}, \eqref{eqLem3.9.3}, we arrive at the formula of the lemma. 
\end{proof}

\begin{prop}\label{prop3.10tce}
	$$\frac{d}{t_\fP}(q-1)c_\fP\equiv (q-1)e_\fP \mod  (|\fp|-1).$$ 
\end{prop}
\begin{proof}
By Corollary \ref{cor2.3}, $\Nr_{\F_\fp^{(d)}/\F_\fp}(\varrho_{\phi, \fp})^{q-1}=\chi_{C, \fp}^{q-1}$. Hence 
$$
\Nr_{\F_\fp^{(d)}/\F_\fp}(\varrho_{\phi, \fp}(\omega_\fP(u)))^{q-1}=\chi_{C, \fp}(\omega_\fP(u))^{q-1}. 
$$	
On one hand, by Lemma \ref{lemrfP}, 
\begin{align*}
\Nr_{\F_\fp^{(d)}/\F_\fp}(\varrho_{\phi, \fp}(\omega_\fP(u)))^{q-1} & = \Nr_{\F_\fp^{(d)}/\F_\fp}(\Nr_{\F_\fP/\F_\fp^{(t_\fP)}}(\bar{u}))^{-c_\fP(q-1)}\\ 
& = \Nr_{\F_\fP/\F_\fp}(\bar{u})^{-\frac{d}{t_\fP}(q-1)c_\fP}. 
\end{align*}
On the other hand, by Lemma \ref{lemchi_fP}, 
$$
\chi_{C, \fp}(\omega_\fP(u))^{q-1}=\Nr_{\F_\fP/\F_\fp}(\bar{u})^{-e_\fP(q-1)}. 
$$
Combining these two expressions, we get 
$$
\Nr_{\F_\fP/\F_\fp}(\bar{u})^{-e_\fP(q-1)} = \Nr_{\F_\fP/\F_\fp}(\bar{u})^{-\frac{d}{t_\fP}(q-1)c_\fP}  
$$
for all $\bar{u}\in \F_\fP^\times$.  Since the norm $\Nr_{\F_\fP/\F_\fp}: \F_\fP^\times\to \F_\fp^\times$ is surjective and $\F_\fp^\times$ 
is cyclic, it follows that $\frac{d}{t_\fP}(q-1)c_\fP\equiv (q-1)e_\fP \mod  (|\fp|-1)$, as was claimed. 
\end{proof}

\begin{cor}\label{cor3.10}
	Let $\fq\lhd A$ be a prime different from $\fp$. Then 
	$$
	r_{\phi, \fp}(\fP)(\fq^{-1})^{d(q-1)}\equiv \fq^{e_\fP f_\fP(q-1)} \Mod{\fp}. 
	$$
\end{cor}
\begin{proof}
	By Lemma \ref{lemrfP}, 
	$$
	r_{\phi, \fp}(\fP)(\fq^{-1})^{d(q-1)}=\Nr_{\F_\fP/\F_\fp^{(t_\fP)}}(\bar{\fq}^{-1})^{-c_\fP d(q-1)}= \Nr_{\F_\fP/\F_\fp^{(t_\fP)}}(\bar{\fq})^{c_\fP d(q-1)}.
	$$
	Since the image $\bar{\fq}$ of $\fq$ under $O_\fP^\times\to \F_\fP^\times$ lies in $\F_\fp^\times$, we have 
	$$
	\Nr_{\F_\fP/\F_\fp^{(t_\fP)}}(\bar{\fq})^{c_\fP d(q-1)}=\bar{\fq}^{c_\fP\frac{f_\fP}{t_\fP}d(q-1)}. 
	$$
	Finally, by Proposition \ref{prop3.10tce}, in $\F_\fp^\times$ we have the equality 
	$$
	\bar{\fq}^{c_\fP\frac{f_\fP}{t_\fP}d(q-1)}=\bar{\fq}^{e_\fP f_\fP (q-1)}. 
	$$
\end{proof}


\section{Drinfeld-Stuhler modules over finite fields}\label{sDSMoverFF}

 Let $\fy\lhd A$ be a prime. Let $k$ be a field 
extension of $\F_\fy$ of degree $m$. Hence $k$ is a finite field of order $q^n$, where $n=m\cdot \deg(\fy)$. 
Let $\pi=\tau^n$ be the associated Frobenius morphism. With abuse of notation, denote by $\pi$ 
also the scalar matrix $\pi I_d\in M_d(k[\tau])$. Note that $\pi$ is in the center of $M_d(k[\tau])$ 
since $\tau^n\alpha=\alpha\tau^n$ for all $\alpha\in k$. 

Let  $\gamma: A\to k$ be the composition 
$A\to A/\fy\hookrightarrow k$; in particular, $\chr_A(k)=\fy$. 
Let $\phi$ be a Drinfeld-Stuhler $O_D$-module 
defined over $k$. In this section we assume that  
 $\fy\not\in \Ram(D)$. Since $\pi$ commutes with $\phi(O_D)$, we have $\pi\in \End_k(\phi)$.

\begin{thm}\label{thm28}  Let $\tF:=F(\pi)$ be the subfield of $D':=\End_k(\phi)\otimes_A F$ generated over $F$ by $\pi$. Then: 
	\begin{enumerate}
		\item $[\tF:F]$ divides $d$. 
		\item There is a unique place $\widetilde{\infty}$ of $\tF$ over $\infty$.
		\item 
		There is a unique prime $\tilde{\fy}\neq \widetilde{\infty}$ of $\tF$ that divides $\pi$. Moreover, $\tilde{\fy}$ lies above $\fy$.  
		\item $D'$ is a central division algebra over $\tF$ of dimension $(d/[\tF:F])^2$ and with invariants 
		$$
		\inv_{\tilde{v}}(D')=\begin{cases}
		-[\tF:F]/d & \text{if $\tilde{v}=\widetilde{\infty}$,}\\
		[\tF:F]/d & \text{if $\tilde{v}=\tilde{\fy}$,} \\ 
		-[\tF_{\tilde{v}}:F_v]\cdot \inv_v(D) & \text{otherwise,}
		\end{cases}
		$$
		for each place $v$ of $F$ and each place $\tilde{v}$ of $\tF$ dividing $v$. 
			\item $\tF$ embeds into $D$. 
	\end{enumerate}
\end{thm}
\begin{proof}
	See \cite[(9.10)]{LRS} and \cite[Thm. 5.1]{PapRMS}. Note that (5) is not explicitly stated in previous references, but it can be easily deduced 
	from (4) as follows. Indeed, (4) implies that for each place $v$ of $F$ and each place $\tilde{v}$ of $\tF$ dividing $v$, we have 
	$$
	d\frac{[\tF_{\tilde{v}}:F_v]}{[\tF:F]}\cdot \inv_v(D)\in \Z. 
	$$
	Now the fact that $\tF$ embeds into $D$ follows from a well-known characterization of commutative subfields of central simple algebras; 
	see, for example, Corollary A.3.4 in \cite{LaumonCDV}. 
\end{proof}

Let $\fl\lhd A$ be a prime different from $\fy$. By \eqref{eqTfpphi}, we have $T_\fl(\phi)\cong O_D\otimes A_\fl$. Thus, 
$T_\fl(\phi)\otimes_{A_\fl} F_\fl\cong D_\fl$.  
The action of $\rG_k$ on torsion points of $\phi$ gives an $\fl$-adic representation 
$$
i_\fl:  \rG_k \to \Aut_{O_D}(T_\fl(\phi)\otimes_{A_\fl} F_\fl) \cong D_\fl^\times. 
$$
Let $\Fr_k\in \rG_k$ be the Frobenius automorphism $\alpha\mapsto \alpha^{q^n}$ of $\bar{k}$. 
Let 
\begin{equation}\label{eqdefP(X)}
P_{\phi, k}(X)=\Nr_{D_\fl/F_\fl}(X-i_\fl(\Fr_k))
\end{equation}
be the reduced characteristic polynomial of $i_\fl(\Fr_k)$. 

\begin{prop}\label{propP=M} 
	Let $M_{\phi, k}(X)\in A[X]$ be the minimal polynomial of $\pi\in \End_k(\phi)$ over $A$. Then 
	$$P_{\phi, k}(X)=M_{\phi, k}(X)^{d/[\tF:F]}.$$ 
	In particular, $P_{\phi, k}(X)$ has coefficients in $A$ that are independent of $\fl$.
\end{prop}
\begin{proof} This can be proved by an argument similar to the argument in the proof of the corresponding fact 
	for Drinfeld modules; cf. \cite[Lem. 3.3]{GekelerFDM}. For the sake of completeness, we give that argument 
	in the setting of Drinfeld-Stuhler modules. 
	
	The endomorphisms of $\phi$ act on $T_\fl(\phi)$, and, by definition of endmorphisms, the actions of $\End_k(\phi)$ and $\phi(O_D)$ on  $T_\fl(\phi)$ commute with each other. Since any nonzero endomorphism has finite kernel, the 
	associated homomorphism 
	$$
	j_\fl: \End_k(\phi)\otimes A_\fl \To \End_{O_D\otimes A_\fl}(T_\fl(\phi))
	$$
	is injective. Hence we get an injective homomorphism 
	$$
	j_\fl: D'\otimes_F F_\fl\To \End_{D_\fl}(T_\fl(\phi)\otimes_{A_\fl} F_\fl)\cong D_\fl. 
	$$
	
	Let $\cN: D'\to F$ be the map obtained by compositing the reduced norm $\Nr_{D'/\tF}: D'\to \tF$ with the field 
	norm $\Nr_{\tF/F}: \tF\to F$. Let $L$ be a maximal commutative subfield of $D'$; in particular, $L$ is a field extension of $F$ 
	of degree $d$. The restriction of $\cN$ to $L$
	agrees with the field norm $\Nr_{L/F}$. On the other hand, $j_\fl(L\otimes_F F_\fl)$ is a maximal commutative 
	$F_\fl$-subalgebra of $D_\fl$ whose norm mapping to $F_\fl$ is given by the reduced norm on $D_\fl$. 
	Therefore, $\cN|_L=(\Nr_{D_\fl/F_\fl}\circ j_\fl)|_L$ for every maximal commutative subfield $L$ of $D'$, so 
	$\cN=\Nr_{D_\fl/F_\fl}\circ j_\fl$.
	
	To prove the claim of the proposition, since $\tF$ is an infinite set, it suffices to show that 
	$P_{\phi, k}(\alpha)=M_{\phi, k}(\alpha)^{d/[\tF:F]}$ for all $\alpha\in \tF$.
	The action $\Fr_k$ on $T_\fl(\phi)\otimes_{A_\fl} F_\fl$ agrees with the action of 
	the endomorphism $\pi \in \End_k(\phi)$, so we have 
	\begin{align*}
	P_{\phi, k}(\alpha) = \Nr_{D_\fl/F_\fl}(\alpha-j_\fl(\pi)) & =\Nr_{\tF/F}\circ \Nr_{D'/\tF}(\alpha-\pi) \\ 
	& =(\Nr_{\tF/F}(\alpha-\pi))^{d/[\tF:F]}\\
	& =M_{\phi, k}(\alpha)^{d/[\tF:F]},
\end{align*}
the last equality coming from $\tF=F(\pi)$. 
\end{proof}

\begin{prop}\label{propP(0)}
	The ideal generated by $P_{\phi, k}(0)$ in $A$ is $\fy^{[k:\F_\fy]}$. 
\end{prop}
\begin{proof}
As follows from Proposition \ref{propP=M}, 
the constant term of $P_{\phi, k}(X)$, up to an $\F_q^\times$ multiple, is equal to $\Nr_{\tF/F}(\pi)^{d/[\tF:F]}$. 
On the other hand, by Theorem \ref{thm28}, the only prime divisor of $\Nr_{\tF/F}(\pi)$ in $A$ is $\fy$. 
Thus,  $(P_{\phi, k}(0))=\fy^s$ for some $s\geq 0$. The claim that $s=[k:\F_\fy]$ follows from the following fact: 
\begin{align*}
\text{(RH) } &\text{ If $|\cdot|_\infty$ 
is the normalized absolute value of $\tF$ corresponding to $\tinf$}, \\ &\text{ then $|\pi|_\infty=(\# k)^{1/d}$.}
\end{align*}
Indeed, since $\tinf$ is the unique place of $\tF$ over $\infty$, the minimal polynomial $M_{\phi, k}(X)$ is irreducible 
over $\Fi$, so all its roots have the same absolute value; combined with (RH), this implies that 
$\deg(P_{\phi, k}(0))=\deg(\fy)\cdot [k:\F_\fy]$. Hence $s=[k:\F_\fy]$. 

(RH) is, of course, the analogue of the Riemann hypothesis for Drinfeld-Stuhler modules. 
As the Anderson motive associated to a Drinfeld-Stuhler module is pure 
of weight $1/d$, 
(RH) follows from a more general statement for Anderson $T$-modules \cite[Thm. 5.6.10]{Goss}. 
\end{proof}

\begin{cor}\label{corP(X)overF_y}
	Assume $k=\F_\fy$. Then $P_{\phi, k}(X)=M_{\phi, k}(X)$ and $\End_k(\phi)\otimes_A F=\tF$ is an 
	imaginary field extension of $F$ of degree $d$. Moreover, if we write 
	$$
	P_{\phi, k}(X)=X^d+a_1X^{d-1}+\cdots+a_d
	$$
	then $\deg(a_i)\leq i\cdot \deg(\fy)/d$ for all $1\leq i\leq d$, and $a_d=\mu\fy$ for some $\mu\in \F_q^\times$. 
\end{cor}
\begin{proof}
By Proposition \ref{propP(0)}, if $k=\F_\fy$, then the constant term of $P_{\phi, k}(X)$ is irreducible in $A$. 
Since $P_{\phi, k}(0)=M_{\phi, k}(0)^{d/[\tF:F]}$, we conclude that $[\tF:F]=d$. Now the equality $\End_k(\phi)\otimes_A F=\tF$
easily follows from Theorem \ref{thm28}. The claim about the degrees of $a_i$ is a consequence of (RH). 
\end{proof}

Since 
$P_{\phi, k}(X)$ is a polynomial with coefficients in $A$, we can reduce the coefficients $P_{\phi, k}(X)$ modulo $\fp$ for any prime $\fp\lhd A$. 

\begin{prop}\label{propP(X)cong} Assume $\inv_\fp(D)=1/d$. Then  
	$$
	P_{\phi, k}(X)\equiv \prod_{i=0}^{d-1}\left(X-\varrho_{\phi, \fp}(\Fr_k)^{|\fp|^i}\right)\mod \fp. 
	$$
\end{prop}
\begin{proof}
	If in \eqref{eqdefP(X)} we take $\fl=\fp$, then $P_{\phi, k}(X)$ modulo $\fp$ is equal to  $\det\left(X-\pi_{\phi, \fp}(\Fr_k)\right)$, 
	where $\pi_{\phi, \fp}$ is the representation from \eqref{eqpi(g)def}. The claim now follows from \eqref{eqpi(g)}. 
\end{proof}

\section{Global points on Drinfeld-Stuhler varieties}\label{sGlobPoints}

In this section we assume that $D$ is a central division algebra over $F$ of dimension $d^2$, $d\geq 2$. 
The goal will be to 
use the canonical isogeny characters to prove that for certain degree $d$ extensions $K$ of $F$ the set of 
$K$-rational points $X^D(K)$ of the Drinfeld-Stuhler variety $X^D$ is empty. 

\subsection{Fields of moduli vs fields of definition}  
A natural approach to proving that $X^D(K)$ is empty for a specific $A$-field $K$ is to show that there are no 
Drinfeld-Stuhler $O_D$-modules defined over $K$. But a technical issue arises in this approach from the fact 
$X^D$ is only a coarse moduli scheme, so the points in $X^D(K)$ may not be represented by 
Drinfeld-Stuhler $O_D$-modules defined over $K$, i.e., the field of moduli of a Drinfeld-Stuhler $O_D$-module
might not be a field of definition. Fortunately, the results from \cite{PapikianCrelle2} and \cite{PapRMS} resolve 
this issue in certain cases.  

Let $K$ be an $A$-field such that $\chr_A(K)\not\in \Ram(D)$, i.e., either $\chr_A(K)=0$ or $\chr_A(K)$ 
is a prime of $A$ which does not ramify in $D$. 
 Let $\phi$ be a 
Drinfeld-Stuhler $O_D$-module defined $K$. The composition $\partial\circ \phi$ gives a homomorphism 
$\partial_\phi: O_D\to M_d(K)$, which extends linearly to a homomorphism 
$$
\partial_{\phi, K}: O_D\otimes_A K \to M_d(K). 
$$ 
It is not hard to show that $\partial_{\phi, K}$ is in fact an isomorphism; see \cite[Lem. 2.5]{PapRMS}. 
This implies that if a $K$-rational point on $X^D$ corresponds to a Drinfeld-Stuhler 
module defined $K$, then necessarily $O_D\otimes_A K\cong M_d(K)$. (If $\chr_A(K)=0$ this is equivalent to $D\otimes_F K\cong M_d(K)$, 
so $K$ splits $D$.) Conversely, we have the following: 

\begin{thm}\label{thmFMFD}
	Let $K$ be an $A$-field such that $\chr_A(K)\not\in \Ram(D)$. If $O_D\otimes_A K\cong M_d(K)$, then  
	a $K$-rational point on $X^D$ corresponds to a Drinfeld-Stuhler $O_D$-module defined over $K$. 
\end{thm}
\begin{proof}
See	\cite[Cor. 6.17]{PapRMS}. 
\end{proof}

Theorem \ref{thmFMFD} does not rule out the possibility that $X^D(K)\neq \emptyset$ but $O_D\otimes_A K\not\cong M_d(K)$, in which  
case the $K$-rational points on $X^D$ would not correspond to Drinfeld-Stuhler $O_D$-modules defined over $K$. 
This unpleasant phenomenon does occur for Shimura curves associated with indefinite quaternion division 
algebras over $\Q$. For example, if $X^6$ denotes the Shimura curve associated to the indefinite quaternion algebra $B(6)$
over $\Q$ of discriminant $6$, then $X^6(\Q(\sqrt{-7}))\neq \emptyset$ although $\Q(\sqrt{-7})$   does not split $B(6)$; see 
Example 1.2 in \cite{Jordan}. On the other hand, for Drinfeld-Stuhler curves we have the following: 

\begin{thm}\label{thmFMFDcurves} 
	Assume $K$ is a finite extension of $F$, $d=2$, 
	and $O_D\otimes_F K\not\cong M_d(K)$. Then  
	$X^D(K)=\emptyset$. 
\end{thm}
\begin{proof} 
	Since $D\otimes_F K\not\cong M_d(K)$, there is at least one prime $\fP$ of $K$ lying over some $\fp\in \Ram(D)$ 
	such that the ramification index $e(\fP|\fp)$ and the residual degree $f(\fP|\fp)$ are both odd (this 
	follows from \eqref{eqChnageofinv}). Let  $K_\fP$ be the completion of $K$ at $\fP$. 
	By \cite[Thm. 4.1]{PapikianCrelle2}, we have $X^D(K_\fP)=\emptyset$. Since $X^D(K)\subset X^D(K_\fP)$, the claim follows. 
\end{proof}

\begin{rem} Another result relevant to the above discussion is the following (see \cite[Cor. 6.17]{PapRMS}): 
	Assume $d$ and $q^d-1$ are coprime, and $\chr_A(K)\not\in \Ram(D)$. If $O_D\otimes_A K\not\cong M_d(K)$, then $X^D(K)=\emptyset$. 
	Because of this fact and Theorems \ref{thmFMFD}, \ref{thmFMFDcurves}, we are inclined to believe that 
	a $K$-rational point on $X^D$ always corresponds to a Drinfeld-Stuhler $O_D$-module defined over $K$. 
\end{rem}

\subsection{Class group conditions}
Let $K/F$ be an extension of degree $d$, and $\fp\lhd A$ be a prime which remains inert in $K/F$. 
(In general, such a prime might not exist if $d>2$.) 
Denote by $\fP$ the unique place of $K$ over $\fp$, and $\infty_1, \dots, \infty_s$ the places of $K$ over $\infty$. 
For a place $v$ of $K$, denote by $K_v$ the completion of $K$ at $v$, and by $O_v$ the ring of integers of $K_v$. 
By abuse of notation, 
let $v$ also denote a uniformizer of $O_v$. Let $\Cl_K$ be the divisor class group of $K$, and 
$\Cl^{\fp\infty}_K$ be the ray divisor class group of $K$ of modulus $\fP\cdot \infty_1\cdots\infty_s$. 
Note that $\Cl_K$ is isomorphic to a direct product of a finite group and $\Z$. 
Let $\mu(K)$ be the group of roots of unity in $K$; 
these are just the nonzero elements of the algebraic closure of $\F_q$ in $K$. There is an exact sequence 
$$
1\To \mu(K)\to (O_{\fP}/\fP)^\times\times \prod_{i=1}^s(O_{\infty_i}/\infty_i)^\times\To \Cl^{\fp\infty}_K\To \Cl_K\To 1. 
$$

\begin{rem}
	If we assume $\infty$ is totally ramified in $K/F$,  
	then the previous exact sequence simplifies to 
	\begin{equation}\label{eqSimplifiedCP}
		1\To (\F_\fp^{(d)})^\times \To \Cl^{\fp\infty}_K\To \Cl_K\To 1. 
	\end{equation}
\end{rem}

Assume 
\begin{itemize}
	\item $\inv_\fp(D)=1/d$, 
	\item $D\otimes_F K\cong M_d(K)$, 
	\item $X^D(K)\neq \emptyset$. 
\end{itemize}
By Theorem \ref{thmFMFD}, there is a Drinfeld-Stuhler $O_D$-module $\phi$ defined over $K$. Let 
$\varrho_{\phi, \fp}: \rG_K\to (\F_\fp^{(d)})^\times$ be the canonical isogeny character constructed in Section \ref{sCanChar}. 
Denote by $Z(n)$ the cyclic group of order $n$. 
By Proposition \ref{lem4.5}, for a place $v\not\in \{\fP, \infty_1, \dots, \infty_s\}$ of $K$, 
the image of the inertia group $I_v$ under $\varrho_{\phi, \fp}$ has order dividing $q^d-1$. Hence 
\begin{equation}\label{eqVarRhoC}
	\varrho_{\phi, \fp}^{q^d-1}: \Gal(K^\ab/K)\to Z\left(\frac{|\fp|^d-1}{q^d-1}\right)
\end{equation}
is unramified outside of $\fP$ and $\infty_1, \dots, \infty_s$. 

\begin{lem} The homomorphism \eqref{eqVarRhoC} restricted to the inertia group at $\fP$ is surjective. 
\end{lem}
\begin{proof}
	Let $K_\fP$ be the completion of $K$ at $\fP$, and $L$ be the totally tamely ramified extension of $K_\fP$ 
	of degree $q^d-1$. It is enough to show that the extension $L(\cC_{\phi, \fp})/L$ is totally ramified of degree $(|\fp|^d-1)/(q^d-1)$. 
	By Proposition \ref{lem4.5}, we can assume that $\phi$ has good reduction over $L$.  
	Let $R$ be the ring of integers of $L$, and $\varpi$ be a uniformizer of $R$ such that $\gamma(\fp)=\varpi^{q^d-1}$. 
	Since the residue field at $\fP$ is $\F_\fp^{(d)}$, the field $L$ contains $\F_\fp^{(d)}$ as a subfield. 
	Now the claim follows from Proposition \ref{lem:Sec2lem2.5}. 
\end{proof}

We may view $\varrho_{\phi, \fp}^{q^d-1}$ as having source $\Cl^{\fp\infty}_K$ 
(note that the ramification at $\fP$ and $\infty_1,\dots, \infty_s$ is tame).  The previous lemma gives a surjective homomorphism
$$
\varrho_{\phi, \fp}^{q^d-1}: \Cl^{\fp\infty}_K\twoheadrightarrow Z\left(\frac{|\fp|^d-1}{q^d-1}\right)   
$$
corresponding to an abelian extension of $K$ which is totally tamely ramified at $\fP$. This extension 
is linearly disjoint from the unramified extensions of $K$ classified by the subgroups of $\Cl_K$. Therefore, applying class field theory, 
we arrive at the following:

\begin{thm}\label{thmGlobalPoints}
	Let $K/F$ be a field extension of degree $d$ which splits $D$. Assume 
	$\fp\lhd A$ is a prime which remains inert in $K/F$ and $\inv_\fp(D)=1/d$. 
	If $X^D(K)\neq \emptyset$, then there is a surjection 
	$$
	\Cl^{\fp\infty}_K\twoheadrightarrow Z\left(\frac{|\fp|^d-1}{q^d-1}\right) \times \Cl_K. 
	$$
\end{thm}

\begin{example}\label{example5.4} Let $q=3$, $d=2$, and $K$ be the quadratic extension $F(\sqrt{\fd})$ of $F$, where  $\fd=T^{13} + 2T + 1$ 
	(this is an irreducible polynomial). We choose   
	$$\fp=T^6 + 2T^4 + T^2 + 2T + 2
	$$
	as a prime of $A$ which remains inert in $K$. 
	Since $\deg(\fd)=13$ is odd, $\infty$ ramifies in $K/F$, so we have the exact sequence \eqref{eqSimplifiedCP}. 
	Next, \texttt{Magma} computes 
	$$
	\frac{|\fp|^2-1}{q^2-1}=2\cdot 5\cdot 7\cdot 13\cdot 73,  
	$$
	\begin{center}
		\begin{tabular}{lll}
			$\Cl_K\cong \Z\times \Z/M\Z,$ & & $M=5^2\cdot 127$, \\ 
			$\Cl^{\fp\infty}_K\cong \Z\times \Z/N\Z,$ & &$N=2^4\cdot 5^3\cdot 7\cdot 13\cdot 73\cdot 127.$
		\end{tabular}
	\end{center} 
	Looking at the $5$-primary components of the groups involved, we see that $\Cl^{\fp\infty}_K$ cannot surject onto 
	$Z\left(\frac{|\fp|^2-1}{q^2-1}\right) \times \Cl_K$. By Theorem \ref{thmGlobalPoints}, if $K$ splits $D$ and $\fp\in \Ram(D)$, then 
	$X^D(K)=\emptyset$. 
\end{example}


\subsection{Congruence conditions} 
Let $K$ be an extension of $F$ of degree $d$ such that  $D\otimes_F K\cong M_d(K)$.  
Assume $\fy\lhd A$ is a prime that totally ramifies in $K$. (In general, such a prime need not exist if $d>2$.) Assume  
$\fy\not \in \Ram(D)$. 
Let $\fY$ be the prime of $K$ lying over $\fy$, so $\fY^d=\fy$. 
Assume $\fp\in \Ram(D)$ and $\inv_\fp(D)=1/d$. Due to $D\otimes_F K\cong M_d(K)$, the prime $\fp$ does not split in $K$. 
Denote by $\fP$ the unique prime of $K$ lying over $\fp$. Denote 
$$n=d\frac{|\fp|^d-1}{|\fp|-1}(q^d-1).$$

Assume there is a Drinfeld-Stuhler  $O_D$-module $\phi$ 
defined over $K$. 
Consider the character $\varrho_{\phi, \fp}^n: \rG_K\to (\F_\fp^{(d)})^\times$. 
Note that $\varrho_{\phi, \fp}^n$ takes values in $\F_\fp^\times$, and is independent of the choice 
of a canonical isogeny character.
By Proposition \ref{propUnramPinf} (or rather by its proof), $\varrho_{\phi, \fp}^n$ is unramified at $\fY$, so 
 $\varrho_{\phi, \fp}^n(\Fr_\fY^d)$ is well-defined, where $\Fr_\fY$ is a Frobenius automorphism at $\fY$. We have the following key congruence:

\begin{prop}\label{prop5.1}
	$\varrho_{\phi, \fp}^n(\Fr_\fY^d)\equiv \fy^{n}\Mod{\fp}$. 
\end{prop}
\begin{proof}
By class field theory, one can consider $\varrho_{\phi, \fp}^n$	as a character of the id\`ele class group $\A_K^\times/K^\times$ of $K$.  
Then 
$$
\varrho_{\phi, \fp}^n(\Fr_\fY^d)=\varrho_{\phi, \fp}^n((\varpi_{\fY}^d)_\fY, (1)^{\fY})=\varrho_{\phi, \fp}^n((u \fy)_\fY, (1)^{\fY})
=\varrho_{\phi, \fp}^n((u)_\fY, (\fy^{-1})^{\fY}),  
$$
where $\varpi_{\fY}$ is a uniformizer of $K_\fY$; $u$ is a unit in $O_\fY$; 
$((\varpi_{\fY}^d)_\fY, (1)^{\fY})$ is the id\`ele of $K$ where the component at the place $\fY$ is $\varpi_{\fY}^d$ and $1$ at all 
other places; $((u)_\fY, (\fy^{-1})^{\fY})$ is the id\`ele where the component at the place $\fY$ is $u$ and $\fy^{-1}$ at all 
other places. Now 
$$
\varrho_{\phi, \fp}^n((u)_\fY, (\fy^{-1})^{\fY})={r}_{\phi, \fp}(\fY)^n(u) \cdot \prod_{v\neq \fY}\widetilde{r}_{\phi, \fp}(v)^n(\fy^{-1}),
$$
where $\widetilde{r}_{\phi, \fp}(v)$ is the character defined in \eqref{eqwidetilder}, and ${r}_{\phi, \fp}(v)$ 
is its restriction to $O_v^\times$. Note that at all places $v$ of $K$ that 
do not divide $\fy$ or $\infty$, the element $\fy^{-1}$ is a unit. Hence by Proposition \ref{propUnramPinf} we have 
$$
{r}_{\phi, \fp}(\fY)^n(u)\cdot \prod_{v\neq \fY}\widetilde{r}_{\phi, \fp}(v)^n(\fy^{-1})=r_{\phi, \fp}(\fP)^n(\fy^{-1}). 
$$
Finally, since $e_\fP f_\fP=[K:F]=d$, Corollary \ref{cor3.10} implies  
$$
r_{\phi, \fp}(\fP)^n(\fy^{-1})\equiv \fy^{e_\fP f_\fP\frac{n}{d}}=\fy^{d\frac{n}{d}}= \fy^n\Mod{\fp}. 
$$
\end{proof}

\begin{defn}
	Let $\cW(\fy)$ be the set of elements $\pi$ of $\overline{F}$ such that 
	\begin{enumerate}
		\item $\pi$ is integral over $A$. 
		\item $[F(\pi):F]=d$. 
		\item There is only one place $\widetilde{\infty}$ of $F(\pi)$ lying over $\infty$. 
		\item There is a unique prime $\tilde{\fy}\neq \widetilde{\infty}$ of $F(\pi)$ that divides $\pi$. This prime lies above $\fy$. 
		\item $|\pi|_\infty=|\fy|^{1/d}$, where $|\cdot|_\infty$ is the unique extension to $F(\pi)$ of the 
		normalized absolute value of $F$ corresponding to $\infty$. 
	\end{enumerate}
\end{defn}

Note that if $\pi\in \cW(\fy)$, then the minimal polynomial of $\pi$ 
$$
M_\pi(X)=X^d+a_{1}X^{d-1}+\cdots+a_d
$$
has the following properties 
\begin{enumerate}
	\item $a_i\in A$ and $\deg(a_i)\leq i \cdot \deg(\fy)/d$ for all $1\leq i\leq d$. 
	\item $a_d=\mu \fy$ for some $\mu\in \F_q^\times$. 
\end{enumerate}
In particular, $\cW(\fy)$ is a finite set. 

For $s\geq 1$, let $n'=d\frac{q^{sd}-1}{q^s-1}(q^d-1)$ and 
$$
\cD'(\fy, s)=\left\{\Nr_{F(\pi)/F}(\pi^{dn'}-\fy^{n'}) \mid \pi\in \cW(\fy)\right\}. 
$$
Let $\cP'(\fy, s)$ be the set of prime divisors of nonzero elements of $\cD'(\fy, s)$. 

 \begin{thm}\label{thmMain1}
	Let $K/F$ be a field extension of degree $d$. Assume
		\begin{itemize}
		\item $D\otimes K\cong M_d(K)$, 
		\item $\fy\lhd A$ is a prime that totally ramifies in $K$, 
		\item $\fy\not\in \Ram(D)$, 
		\item $\inv_\fp(D)=1/d$, 
		\item $\fp\not\in \cP'(\fy, \deg(\fp))$, 
		\item $D\otimes_F F(\sqrt[d]{\mu \fy})\not\cong M_d(F(\sqrt[d]{\mu \fy}))$ for any $\mu\in \F_q^\times$. 
	\end{itemize}
	Then $X^D(K)=\emptyset$. 
\end{thm}
\begin{proof}
	Suppose $X^D(K)\neq\emptyset$. Since $D\otimes K\cong M_d(K)$, Theorem \ref{thmFMFD} implies that there 
	exists a Drinfeld-Stuhler $O_D$-module $\phi$ defined over $K$.  
	Consider $\phi$ over $K_\fY$. By Proposition \ref{lem4.5}, $\phi$ has good reduction over the totally tamely ramified extension $L$ of $K_\fY$ 
	of degree $q^d-1$. Denote by $\bar{\phi}$ the reduction of $\phi$ modulo the 
	maximal ideal of $L$. Note that the residue field of $L$ is $\F_\fy$. 
	By Theorem \ref{thm28} and (RH) in the proof of Proposition \ref{propP(0)}, the roots of $P_{\bar{\phi}, \F_\fy}(X)$ 
	are in $\cW(\fy)$. If we decompose 
	$$
	P_{\bar{\phi}, \F_\fy}(X)=\prod_{i=1}^d(X-\pi_i)
	$$
	over $\overline{F}$, then from \eqref{eqdefP(X)} it is easy to see that 
	$$
	P_{\bar{\phi}, \F_\fy^{(dn)}}(X)=\prod_{i=1}^d(X-\pi_i^{dn}). 
	$$
	On the other hand, by Proposition \ref{propP(X)cong} 
	$$
	P_{\bar{\phi}, \F_\fy^{(dn)}}(X)\equiv \prod_{i=0}^{d-1}\left(X-\varrho_{\bar{\phi}, \fp}(\Fr_{\F_\fy}^{dn})^{|\fp|^i}\right)=
	\prod_{i=0}^{d-1}\left(X-\varrho_{\bar{\phi}, \fp}(\Fr_{\F_\fy}^{dn})\right)\mod{\fp}, 
	$$
	where the second equality follows from the fact that $\varrho_{\bar{\phi}, \fp}^n$ takes values in $\F_\fp^\times$. 
	
	Since $L$ is totally ramified over $F_\fy$, we have $\varrho_{\phi, \fp}(\Fr_\fY^{nd})= \varrho_{\bar{\phi}, \fp}(\Fr_{\F_\fy}^{nd})$. 
	By Proposition \ref{prop5.1}, 
	$$\varrho_{\phi, \fp}(\Fr_\fY^{nd})\equiv \fy^{n}\Mod{\fp}.$$ 
	Thus, 
	$$
	\prod_{i=1}^d(X-\pi_i^{dn})\equiv \prod_{i=1}^d(X-\fy^n) \mod{\fp}. 
	$$
	Since $\F_\fp$ is a field, this congruence implies $\pi_i^{dn}\equiv \fy^n\Mod{\fp}$ for all $1\leq i\leq d$, so 
	$\fp$ divides $\Nr_{F(\pi_i)/F}(\pi_i^{dn}- \fy^n)$. This contradicts the assumption $\fp\not\in \cP'(\fy, \deg(\fp))$, unless $\pi_i^{dn}=\fy^n$ 
	for all $i$. On the other hand, if $\pi_i^{dn}=\fy^n$, then $\bar{\phi}$ is supersingular and there is a unique place of $\tF$ 
	lying over $\fy$; see \cite[Prop. 5.3]{PapRMS}. By Corollary \ref{corP(X)overF_y}, the polynomial $P_{\bar{\phi}, \F_\fy}(X)$ 
	is irreducible over $F$, and, since only one place of $\tF$ lies over $\fy$, it  
	remains irreducible over the completion $F_{\fy}$ of $F$ at $\fy$. Hence the Newton polygon of $P_{\bar{\phi}, \F_\fy}(X)$  
	over $F_{\fy}$ must have only one slope.  Let $P_{\bar{\phi}, \F_\fy}(X)=X^d+a_1X^{d-1}+\cdots+a_d$. 
	By Corollary \ref{corP(X)overF_y}, $a_d=-\mu\fy$ for some $\mu\in \F_q^\times$ and any other coefficient $a_i$ is not divisible 
	by $\fy$, unless it is zero. If $a_i\neq 0$ for some $1\leq i\leq d-1$, then the Newton polygon of $P_{\bar{\phi}, \F_\fy}(X)$ 
	has two slopes, contradicting our earlier conclusion.  
	Thus, $P_{\bar{\phi}, \F_\fy}(X)=X^d-\mu\fy$ for some $\mu\in \F_q^\times$. But now, in the notation of Theorem 
	\ref{thm28}, the field $\tF=F(\sqrt[d]{\mu\fy})$ embeds into $D$. 
	We conclude that 
	in our case $\tF$ is a maximal commutative subfield of $D$, so $\tF$ splits $D$.
	This contradicts one of our assumptions. Therefore $X^D(K)=\emptyset$. 
\end{proof}

\begin{example}
	Let $q=3$, $d=2$, $s=3$, and $\fy=T$. A computer calculation using \texttt{Magma} shows that 
	$T^3 + T^2 + 2$ and $T^3 + 2T^2 + 1$ are not in $\cP'(\fy, 3)$. Let $\fp$ be one of these primes. Let $\fq$ 
	be a prime different from $\fp$ which splits in both $F(\sqrt{T})$ and $F(\sqrt{-T})$. Let $D$ 
	be the quaternion division algebra over $F$ with $\Ram(D)=\{\fp, \fq\}$.  Let $K=F(\sqrt{T\fm})$, where $\fm$ 
	is square-free and $T\nmid \fm$. If $\fm$ is chosen so that $\Leg{T\fm}{\fp}\neq 1$ and $\Leg{T\fm}{\fq}\neq 1$, then 
	$K$ splits $D$. For example, one can take $\fq=T^2+1$ and $\fm=\fp\fq$. Therefore, by Theorem \ref{thmMain1}, $X^D(K)=\emptyset$. 
\end{example}

Let $m=d(q^d-1)$. Fix a canonical isogeny claracter $\varrho_{\phi, \fp}$. Denote by $\bar{\fy}$ the image of $\fy$ under the reduction 
homomorphism $A\to \F_\fp$. By Proposition \ref{prop5.1}, 
$$
\eps:=\varrho_{\phi, \fp}(\Fr_\fY^{d m})/\bar{\fy}^{m}\in \F_\fp^{(d)}
$$
satisfies 
$$
\Nr_{\F_\fp^{(d)}/\F_\fp}(\eps)=\eps^{1+|\fp|+\cdots+|\fp|^{d-1}}=1. 
$$
As in the proof of Theorem \ref{thmMain1} one deduces that 
$$
P_{\bar{\phi}, \F_\fy^{(dm)}}(X)\mod{\fp}=\prod_{i=0}^{d-1}\left(X-\varrho_{\phi, \fp}(\Fr_\fY^{dm})^{|\fp|^i}\right)=\prod_{i=0}^{d-1}\left(X-\eps^{|\fp|^i}\bar{\fy}^m\right)\in \F_\fp[X]
$$
is independent of the choice of a canonical isogeny claracter. We use this expression to strengthen Theorem \ref{thmMain1} when $d=2$. 
This will give us a congruence condition which is more amenable to explicit verification, and thus construction of explicit examples. 

Until the end of this section assume $d=2$. For $r\geq 1$ denote  
$$
P_{\bar{\phi}, \F_\fy^{(r)}}(X)=X^2-a(r)X+b(r)\in A[X]. 
$$ 
For $\pi\in \cW(\fy)$, let $\pi'$ be its conjugate over $F$ (keep in mind that $[F(\pi):F]=d=2$). Define 
$$
\cC(\fy)=\{\pi^{dm}+(\pi')^{dm}\mid \pi\in \cW(\fy)\}. 
$$
Note that $\cC(\fy)$ is a finite subset of $A$, and $a(dm)\in \cC(\fy)$. For $s\geq 1$, let 
$$
\cD(\fy, s)=\left\{ \Nr_{\F_{q^s}(T)/F}\left(c-(\eps+\eps^{-1})\fy^m\right)\mid c\in \cC(\fy), \eps\in \F_{q^{2s}}, \eps^{1+q^s}=1 \right\}. 
$$
Note that $\eps+\eps^{-1}=\eps+\eps^{q^s}=\Tr_{\F_{q^{2s}}/\F_{q^s}}(\eps)\in \F_{q^s}$, so the norm above makes sense. 
$\cD(\fy, s)$ is again a finite subset of $A$. Let $\cP(\fy, s)$ be the set of prime divisors of nonzero elements of $\cD(\fy, s)$. 

Now let $s=\deg(\fp)$ and identify $\F_{q^{2s}}$ with $\F_\fp^{(d)}$. From our earlier discussion, 
$$
a(dm)\equiv (\eps+\eps^{-1}) \bar{\fy}^m\mod{\fp}. 
$$
This implies that either $\fp\in \cP(\fy, s)$ or $a(dm)= (\eps+\eps^{-1}) \fy^m$. In the second case, $\fy$ divides $a(dm)$. 

\begin{lem}\label{lem5.4}
	$\fy$ divides $a(dm)$ if and only if $a(1)=0$. 
\end{lem}
\begin{proof}
	Write $a(1)=\pi+\pi'$ for some $\pi\in \cW(\fy)$. Then $a(r)=\pi^r+(\pi')^r$ for all $r\geq 1$. Since 
	$$
	a(1)^r=(\pi+\pi')^r=\pi^r+(\pi')^r+\pi\pi'c= \pi^r+(\pi')^r + \fy \mu c=a(r)+\fy \mu  c
	$$
	for some $c\in A$ and $\mu\in \F_q^\times$, we see that $\fy$ divides $a(r)$ if and only if $\fy$ divides $a(1)$. 
	On the other hand, $\deg(a(1))\leq \deg(\fy)/2$, so $\fy$ divides $a(1)$ if and only if $a(1)=0$. 
\end{proof}

By Lemma \ref{lem5.4}, if $\fy$ divides $a(dm)$ then $a(1)=0$. But if $a(1)=0$, then the minimal polynomial of 
the Frobenius endomorphism of $\bar{\phi}$ over $\F_\fy$ is $X^2-\mu\fy$ for some $\mu\in \F_q^\times$. 
Thus,  $\tF=F(\sqrt{\mu \fy})$. By Theorem \ref{thm28}, $\tF$ embeds into $D$, so 
$\tF$ splits $D$. Thus, we proved the following: 

\begin{thm}\label{thmMain2} Let $d=2$ and $K/F$ be a quadratic extension. Assume 
	\begin{itemize}
		\item $D\otimes K\cong M_2(K)$, 
		\item $\fy\lhd A$ is a prime that ramifies in $K$, 
		\item $\fy\not\in \Ram(D)$, 
		\item $\fp\in \Ram(D)$, 
		\item $\fp\not\in \cP(\fy, \deg(\fp))$, 
		\item $D\otimes_F F(\sqrt{\mu \fy})\not\cong M_2(F(\sqrt{\mu \fy}))$ for any $\mu\in \F_q^\times$. 
		\end{itemize}
	Then $X^D(K)=\emptyset$. 
\end{thm}

	Note that the last assumption $D\otimes_F F(\sqrt{\mu \fy})\not\cong M_2(F(\sqrt{\mu \fy}))$
	is equivalent to the existence of $\fq\in \Ram(D)$ which splits in $F(\sqrt{\mu \fy})$. 
	Also, $\infty$ should not split in $\tF$, which is equivalent to $\mu$ being a non-square in $\F_q^\times$ when 
	$\deg(\fy)$ is even and $q$ is odd. 

\begin{rem}
	Theorem \ref{thmMain1} specialized to $d=2$ is a weaker theorem than Theorem \ref{thmMain2} because $\cP(\fy, s)$ 
	is a smaller set than $\cP'(\fy, s)$. For example, if $q=3$, $\fy=T$ and $s=2$, then $T^2+T+2$ and $T^2+2T+2$ 
	are not in $\cP(\fy, 2)$ but they are both in $\cP'(\fy, 2)$. 
\end{rem}

\begin{example}\label{example5.9}
	Let $q=3$, $d=2$, $\fy=T$. Let 
	$$
	S:=\{T^2+T+2, \quad T^2+2T+2, \quad T^3+T^2+2, \quad T^3+2T^2+1\}. 
	$$
	A computer calculation shows that if $\fp\in S$, then $\fp\not\in \cP(\fy, \deg(\fp))$. 
	
	Let $\fq$ be a prime different from $\fp$ which splits in both $F(\sqrt{T})$ and $F(\sqrt{-T})$, i.e., 
	$\Leg{\pm T}{\fq}= 1$. For example, one can take $\fq=T^2+1.$ 
	Let $D$ be a quaternion algebra over $F$ with $\Ram(D)=\{\fp, \fq\}$. Since 
	$\fq$ splits in $F(\sqrt{\pm T})$, the last assumption of Theorem \ref{thmMain2} is satisfied, i.e., $D\otimes F(\sqrt{\pm T})
	\not\cong M_2(F(\sqrt{\pm T}))$. 
	
	Let $K=F(\sqrt{T\fm})$, where $\fm\in A$ is a square-free, but not necessarily monic, polynomial 
	not divisible by $T$. If $\fm$ is chosen so that $\Leg{T\fm}{\fp}\neq 1$ and $\Leg{T\fm}{\fq}\neq 1$, then 
	$K$ splits $D$. (For example, one can take $\fm=\fp\fq\fm_1$, where $\fm_1$ is an arbitrary 
	square-free polynomial coprime to $T\fp\fq$.) With previous choices, Theorem \ref{thmMain2} implies that $X^D(K)=\emptyset$. 
\end{example}

\begin{example}\label{example5.10}
		Let $q=5$, $d=2$, $\fy=T$. Let 
	\begin{align*}
	S:=\{& T^3+2T+4, \quad T^3+3T+3, \quad T^3+T^2+T+4, \quad T^3+T^2+3T+1,  \\ 
	& T^3+2T^2+T+3,  \quad T^3+2T^2+2T+3, \quad T^3+3T^2+2T+3, \\ 
	& T^3+3T^2+4T+3,  \quad  T^3+4T^2+T+1, \quad   T^3+4T^2+3T+4\}. 
	\end{align*}
	A computer calculation shows that if $\fp\in S$, then $\fp\not\in \cP(\fy, \deg(\fp))$. Fix a prime $\fp\in S$. 
	
	Let $\fq\lhd A$ be a prime such that 
	\begin{itemize}
		\item[(i)] $\fq\neq T, \fp$; 
		\item[(ii)] $\deg(\fq)$ is odd; 
		\item[(iii)] $\Leg{\fp}{T}=-\Leg{\fq}{T}$. 
	\end{itemize}
	Let $D$ be a quaternion algebra over $F$ with $\Ram(D)=\{\fp, \fq\}$. We claim that $F(\sqrt{\mu T})$ does not split 
	$D$ for any $\mu\in \F_q^\times$. It is enough to prove this for $\mu=1, 2$ (since $3=2\cdot 4$ and $4=1\cdot 4$). Assume first that   
	$\Leg{\fp}{T}=1$. Then by Quadratic Reciprocity $\Leg{T}{\fp}=\Leg{\fp}{T} (-1)^{\frac{5-1}{2}\deg(T)\deg(\fp)}=1$. 
	Hence $\fp$ splits in $F(\sqrt{T})$, and therefore $D\otimes F(\sqrt{T})\not\cong M_2(F(\sqrt{T}))$. Since $\Leg{\fq}{T}=-1$, 
	we have $\Leg{T}{\fq}=-1$. Since $\deg(\fq)$ is odd and $2\in \F_q^\times$ is not a square, we have $\Leg{2}{\fq}=-1$. Therefore 
	$\Leg{2T}{\fq}=1$, and so $\fq$ splits in $F(\sqrt{2T})$. As before, this implies $D\otimes F(\sqrt{2T})\not\cong M_2(F(\sqrt{2T}))$. 
	If $\Leg{\fp}{T}=-1$, then a similar argument applies. 
	
	Now, as in Example \ref{example5.9}, let $K=F(\sqrt{T\fm})$, where $\fm\in A$ is a square-free, but not necessarily monic, polynomial 
	not divisible by $T$. If $\fm$ is chosen so that $\Leg{T\fm}{\fp}\neq 1$ and $\Leg{T\fm}{\fq}\neq 1$, then 
	$K$ splits $D$. With previous choices, Theorem \ref{thmMain2} implies that $X^D(K)=\emptyset$. 
\end{example}


\section{Counterexamples to the Hasse principle}\label{sHP}

In this section, we extend Examples \ref{example5.4}, \ref{example5.9}, \ref{example5.10} 
to construct explicit examples of curves violating the Hasse principle. The main auxiliary tool that we will use 
are the results from \cite{PapikianCrelle2} on the existence of local points on Drinfeld-Stuhler curves, which 
are the function field analogues of the results of Jordan and Livn\'e for Shimura curves \cite{JLLocal}. 
For the convenience of the reader, we summarize these results specialized to the case that will be of particular interest for us. 

Let $K/F$ be a quadratic extension. Let $D$ be a quaternion algebra over $F$ with $\Ram(D)=\{\fp, \fq\}$, where $\fp$ and 
$\fq$ are two distinct primes. For a place $v$ of $K$ denote by $K_v$ the completion of $K$ at $v$. 

\begin{thm}\label{thmInfBad}
	Let $\tinf$ be a place of $K$ over $\infty$. 
		\begin{enumerate}
			\item If $\infty$ does not split in $K$, then $X^D(K_{\tinf})\neq \emptyset$. 
			\item If $\infty$ splits in $K$, then $X^D(K_{\tinf})=\emptyset$ if and only if either $\deg(\fp)$ or $\deg(\fq)$ is even. 
		\end{enumerate}
\end{thm}
\begin{proof}
	This is a consequence of Theorem 5.10 in \cite{PapikianCrelle2}. 
\end{proof}

\begin{rem}
	Assume $q$ is odd and $K=F(\sqrt{\fd})$ for a square-free polynomial $\fd\in A$. It is easy to show that $\infty$ splits in $K$ 
	if and only if $\deg(\fd)$ is even and its leading coefficient is a square in $\F_q^\times$. 
\end{rem}

\begin{thm}\label{thmFinBad} Let $\fP$ be a place of $K$ over $\fp$. 
	Let $f$ and $e$ be the residual degree and the ramification index of $\fP$ over $\fp$, respectively.  
	\begin{enumerate}
		\item If $f=2$, then $X^D(K_\fP)\neq \emptyset$. 
		\item If $e=2$, then $X^D(K_\fP)= \emptyset$ if and only if in every quadratic extension $F(\sqrt{\mu \fp})$, with 
		$\mu\in \F_q^\times$, at least one of the places $\fq, \infty$ splits. 
		\item If $\fp$ splits in $K$, then $X^D(K_\fP)= \emptyset$. 
	\end{enumerate}
\end{thm}
\begin{proof}
	This is a consequence of Theorem 4.1 in \cite{PapikianCrelle2}. 
\end{proof}

\begin{rem}
	In this section, we are only interested in quadratic extensions $K$ which split $D$. In such 
	extensions, $\fp$ or $\fq$ cannot split, so (3) of Theorem \ref{thmFinBad} does not occur. 
\end{rem}

\begin{thm}\label{thmGoodPlaces} Let $\fl\not\in \{\fp, \fq, \infty\}$ be a place of $F$ and $\fL$ be a place of $K$ over $\fl$. 
	Let $f$ be the residual degree of $\fL$ over $\fl$. 
	\begin{enumerate}
		\item If $f=2$, then $X^D(K_\fL)\neq \emptyset$. 
		\item If $f=1$, then $X^D(K_{\fL})=\emptyset$ 
	if and only if for all $a\in A$ and $c\in \F_q^\times$ such that 
		$a^2+c\fl$ is not a square in $A$ at least one of the places $\{\fp, \fq, \infty\}$ splits in $F(\sqrt{a^2+c\fl})$.
	\end{enumerate}
\end{thm}
\begin{proof}
	This is a consequence of Theorem 3.1 in \cite{PapikianCrelle2}. 
\end{proof}

\begin{rem}\label{rem7.6}
	If $q$ is odd, then to determine whether $X^D(K_\fL)= \emptyset$ one needs to examine only a finite number 
	of quadratic extensions $F(\sqrt{a^2+c\fl})$. Indeed, if $\deg(a)> \deg(\fl)/2$, then 
	$a^2+c\fl$ has even degree and its leading coefficient is a square in $\F_q^\times$, so $\infty$ splits in $F(\sqrt{a^2+c\fl})$. 
	If $q$ is even, then $X^D(K_\fL)\neq \emptyset$ since the places $\{\fp, \fq, \infty\}$ ramify in the inseparable 
	extension $F(\sqrt{\fl})$. 
\end{rem}

Let $\fl$ be as in Theorem \ref{thmGoodPlaces}. By \cite{LRS} and \cite{Hausberger}, $X^D$ has 
good reduction at $\fl$ (and bad reduction at $\fp, \fq$, $\infty$).  
From a geometric version of Hensel's Lemma (see \cite{JLLocal}) it follows that 
$X^D(K_\fL)\neq \emptyset$ if and only if $X^D(k_\fL)\neq \emptyset$, where $k_\fL$ is the residue field at $\fL$. 
If $\fl$ is not inert in $K$ (i.e., $f=1$), then $k_{\fL}=\F_\fl$. Now by the Weil bound 
$$
\# X^D(\F_\fl)\geq |\fl|+1-2g(X^D)\sqrt{|\fl|}, 
$$
where $g(X^D)$ is the genus of $X^D$. On the other hand, by \cite{PapGenus}, 
\begin{equation}\label{eqGenus}
g(X^D)=1+\frac{(|\fp|-1)(|\fq|-1)}{q^2-1}-
\begin{cases}
0 & \text{if  $\deg(\fp)$ or $\deg(\fq)$ is even}\\
\frac{2q}{q+1} & \text{if  $\deg(\fp)$ and $\deg(\fq)$ are odd}
\end{cases}. 
\end{equation}
This implies that $X^D(\F_\fl)\neq \emptyset$ once $\deg(\fl)\geq 2(\deg(\fp)+\deg(\fq))$.  Hence 
to decide whether a given $X^D$ has rational points over all completions of $K$, we need to examine 
only finitely many places. 

\begin{example}\label{example7.7} Let the setting be as in Example \ref{example5.4}. That is 
	$q=3$, $\fp=T^6 + 2T^4 + T^2 + 2T + 2$, and $K=F(\sqrt{\fd})$, where $\fd=T^{13}+2T+1$. 
	
	Next, we take $\fq=T^2 + T + 2$ as a prime of smallest possible degree which is inert in $K$ 
($T$ and $T\pm 1$ split in $K$). Let $D$ be the quaternion algebra ramified at $\fp$ and $\fq$, and split 
	at all other places of $F$. Since $\fp$ and $\fq$ are inert in $K$, we have $D\otimes K\cong M_2(K)$. 
	We will show that $X^D(K_v)\neq \emptyset$ for all places $v$ of $K$. Since, by Example \ref{example5.4},  $X^D(K)=\emptyset$, 
	this gives an explicit example of a Drinfeld-Stuhler curve violating the Hasse principle over $K$. 
	
	The place $\infty$ of $F$ ramifies in $K$. By Theorem \ref{thmInfBad},  we have $X^D(K_{\widetilde{\infty}})\neq \emptyset$. 
	
	The places $\fp$ and $\fq$ are inert in $K$. By Theorem \ref{thmFinBad}, we have 
	$X^D(K_{\fP})\neq \emptyset$ and $X^D(K_{\fQ})\neq \emptyset$ for the unique places $\fP$ and $\fQ$ of $K$ 
	over $\fp$ and $\fq$, respectively. 
	
	It remains to examine places of $F$ where $X^D$ has good reduction. By Theorem \ref{thmGoodPlaces}, we 
	can restrict ourselves to those $\fl$ that split or ramify in $K$. Using \eqref{eqGenus} one computes that $g(X^D)=q^6$.  
	From the Weil bound, $X^D(\F_\fl)\neq \emptyset$ once $\deg(\fl)\geq 13$. 
	Hence we need to examine $X^D(F_\fl)$ for primes $\fl\lhd A$ 
	of degree $\leq 12$ which split in $K$. 
	By Theorem \ref{thmGoodPlaces}, to show that $X^D(F_\fl)\neq \emptyset$  
	we need to find $a\in A$ with $\deg(a)\leq \deg(\fl)/2$ and $c\in \F_q^\times$ such that 
	\begin{itemize}
		\item $a^2+c\fl$ has odd degree or its leading coefficient is not a square in $\F_q^\times$, and 
		\item $$\Leg{a^2+c\fl}{\fp}\neq 1, \qquad \Leg{a^2+c\fl}{\fq}\neq 1.$$ 
	\end{itemize} 
	For this we use computer calculations. (It is important here that places $\fl$ have relatively small degree, 
	which is a consequence of choosing $\fq$ of small degree so that the genus of $X^D$ is small.)  
	Using \texttt{Magma}, we simply run through all possible $a$ and $c$ and check the conditions. 
	It turns out that $a$ and $c$ for which the necessary conditions are satisfied always exist.  
	For example, assume $\deg(\fl)=1$, i.e., $\fl=T, T\pm 1$. All these primes split in $K$, and it is easy to check that for each such prime 
	one can obtain $T$ or $-T$ as 
	$a^2+c\fl$ for some $a\in \F_3$ and $c\in \F_3^\times$. On the other hand, 
	$$
	\Leg{\pm T}{\fp}= \Leg{\pm T}{\fq}=-1. 
	$$
\end{example}

 \begin{example}\label{example7.8}
Let $q=3$, 
$$
\fp\in \{T^2+T+2, \quad T^2+2T+2\}, \qquad \fq=T^2+1,
$$ 	
and $K=F(\sqrt{\fd})$, where $\fd=\pm T\fp\fq\fm$ for some monic square-free polynomial $\fm$ coprime to $T\fp\fq$. 
Further assume that either $\deg(\fm)$ is even or the leading coefficient of $\fd$ is $-1$. Note that 
this last assumption implies that $\infty$ does not split in $K$.  Let $D$ 
be the quaternion algebra over $F$ with $\Ram(D)=\{\fp, \fq\}$. By Example \ref{example5.9}, $X^D(K)=\emptyset$. 

Since $\infty$ does not split in $K$, by Theorem \ref{thmInfBad},  we have $X^D(K_{\widetilde{\infty}})\neq \emptyset$. 

The places $\fp$ and $\fq$ ramify in $K$. By Theorem \ref{thmFinBad}, we have 
$X^D(K_{\fP})\neq \emptyset$ and $X^D(K_{\fQ})\neq \emptyset$ for the unique places $\fP$ and $\fQ$ of $K$ 
over $\fp$ and $\fq$, respectively. (Note that $\fq$ and $\infty$ are inert in $F(\sqrt{-\fp})$, and similarly 
$\fp$ and $\infty$ are inert in $F(\sqrt{-\fq})$.) 

As in Example \ref{example7.7}, we compute $g(X^D)=q^2$ and use the Weil bound to concludes that $X^D(F_\fl)\neq \emptyset$ 
if $\deg(\fl)\geq 6$. Finally, we examine places $\fl$ of degree $\leq 5$ using a computer program. For the following $\fm$ 
our program confirms that the conditions for the existence of $K_\fL$-rational points are satisfied for all $\fl$: 
\begin{align*}
\fm\in \{& 1, \quad T+1, \quad T+2,  \quad T^2+T+2,  \quad T^2+2T+2,  \quad T^3 +T^2 +2,  \\ 
&T^3 +2T^2 +1,    \quad T^3 + 2T+1,  \quad  T^3 +2T+2,  \quad T^3 +T^2 +2T+1, \\ 
& T^3 +T^2 +T+2,     \quad T^4 +T+2,   \quad T^4 +2T+2,   \quad T^4 + T^2 +2, \\ 
 & T^4 +2T^2 +2, \quad  T^4 +T^2 +T+1, \quad  T^4 +T^2 +2T+1\}\\ 
& \text{(where the case $\fm=\fp$ is excluded).}
\end{align*}
Thus, we obtain counterexamples to the Hasse principle. 
 \end{example}

\begin{rem}
	In Example \ref{example5.9} we also had $\fp\in \{T^3+T^2+2, T^3+2T^2+1\}$ and $\fq=T^2+1$ 
as possible ramification places of a quaternion algebra $D$ for which $X^D(K)=\emptyset$, where $K=F(\sqrt{\pm T\fp\fq\fm})$. 
But these $\fp$ do not lead to counterexamples to the Hasse principle since in these cases $X^D(K_\fQ)=\emptyset$. 
This follows from Theorem \ref{thmFinBad}, as $\infty$ splits in $F(\sqrt{\fq})$ and $\fp$ splits in $F(\sqrt{-\fq})$ 
\end{rem}

\begin{example} 
	Let $q=5$, $\fp= T^3+2T+4$, and $\fq=T+2$.  Note that this puts us in the setup of Example \ref{example5.10} since 
	$\deg(\fq)$ is odd and $\Leg{\fp}{T}=1=-\Leg{\fq}{T}$. Let $\fm\in A$ be a monic square-free polynomial coprime to $T\fp\fq$.  
	Let $\fd=T\fp\fq\fm$ or $\fd=2T\fp\fq\fm$ if $\deg(\fm)$ is even, and $\fd=2T\fp\fq\fm$ if $\deg(\fm)$ is odd. 
	Let $D$ be the quaternion algebra over $F$ with $\Ram(D)=\{\fp, \fq\}$. Then, by Example \ref{example5.10}, 
	we have $X^D(K)=\emptyset$ for $K=F(\sqrt{\fd})$.
	
	Since $\infty$ does not split in $K$, we have $X^D(K_{\widetilde{\infty}})\neq \emptyset$. 
	
	The places $\fp$ and $\fq$ ramify in $K$. 
	Since $\deg(\fp)=3$ is odd, $\infty$ ramifies in both $F(\sqrt{\fp})$ and $F(\sqrt{2\fp})$. The prime 
	$\fq$ is inert in $F(\sqrt{\fp})$. Thus, $X^D(K_\fP)\neq \emptyset$ 
	for the unique prime $\fP$ of $K$ over $\fp$. Similarly, $\infty$ ramifies in both $F(\sqrt{\fq})$ and $F(\sqrt{2\fq})$, and 
	$\fp$ is inert in $F(\sqrt{\fq})$. Thus, $X^D(K_\fQ)\neq \emptyset$. 
	
	Next, we compute $g(X^D)=q(q-1)$ and use the Weil bound to conclude that $X^D(F_\fl)\neq \emptyset$ 
	if $\deg(\fl)\geq 5$. Finally, we examine places $\fl$ of degree $\leq 4$ using a computer program. For the following $\fm$ 
	our program confirms that the conditions for the existence of $K_\fL$-rational points are satisfied for all $\fl$: 
	$$
	\fm\in \{ 1, \quad T+1,  \quad T+3,  \quad T+4,  \quad T^2 +2,  \quad T^2 +3,  \quad T^2 +T+1,  \quad T^2 +T+2\}. 
	$$
	Thus, we obtain counterexamples to the Hasse principle. 
 \end{example}

\begin{rem}
	By Theorem \ref{thmFinBad}, for $\fp\in \Ram(D)$ we have $X^D(F_\fp)=\emptyset$. 
	Thus, $X^D(F)=\emptyset$ but $X^D$ does not violate the Hasse principle over $F$. 
\end{rem}

\begin{rem} If $X^D$ embeds into its Jacobian $J^D$ over a finite extension $K$ of $F$ (e.g., there is 
	a $K$-rational divisor of degree $1$ on $X^D$), then 
	the Brauer-Manin obstruction is the only obstruction to the Hasse principle over $K$. 
	This follows from a result of Poonen and Voloch \cite[Thm. D]{PV}. There are two conditions that 
	need to be satisfied in order to apply \cite{PV}. These conditions are 
	\begin{enumerate}
		\item $J^D$ has no nonzero isotrivial quotients; 
		\item $J^D(F^\sep)[p^\infty]$ is finite ($p$ is the characteristic of $F$). 
	\end{enumerate}
	These are satisfied since $J^D$ has split purely multiplicative reduction at $\infty$ (cf. \cite{BS}, \cite{PapikianCrelle2}). 
	Indeed, the first claim is clear from the theory of N\'eron models. The second claim can be seen from the rigid-analytic uniformization 
	$J^D(\overline{F}_\infty)\cong (\overline{F}_\infty^\times)^{g(X^D)}/\La$ of $J^D$ over $\Fi$ -- the 
	extension of $\Fi$ obtained by adjoining $J^D(\overline{F}_\infty)[p^n]$ is the same as the extension 
	obtained by adjoining $p^n$-th roots of the generators of the lattice $\La$.  
\end{rem}

\subsection*{Acknowledgements} 
Keisuke  Arai thanks Shin Hattori for a discussion and suggestive comments. 
Satoshi Kondo thanks K\"ur{\c s}at Aker for a computation and
Shin Hattori in general. 
Mihran Papikian thanks the National Center for Theoretical Sciences in  Hsinchu, where 
part of this work was carried out, for its hospitality and excellent working conditions.

\bibliographystyle{acm}
\bibliography{DSmodules2.bib}

\begin{thebibliography}{10}

\bibitem{Anderson}
{\sc Anderson, G.~W.}
\newblock {$t$}-motives.
\newblock {\em Duke Math. J. 53}, 2 (1986), 457--502.

\bibitem{Arai1}
{\sc Arai, K.}
\newblock Non-existence of points rational over number fields on {S}himura
  curves.
\newblock {\em Acta Arith. 172}, 3 (2016), 243--250.

\bibitem{Arai2}
{\sc Arai, K.}
\newblock Rational points on {S}himura curves and the {M}anin obstruction.
\newblock {\em Nagoya Math. J. 230\/} (2018), 144--159.

\bibitem{BS}
{\sc Blum, A., and Stuhler, U.}
\newblock Drinfeld modules and elliptic sheaves.
\newblock In {\em Vector bundles on curves---new directions ({C}etraro, 1995)},
  vol.~1649 of {\em Lecture Notes in Math.} Springer, 1997, pp.~110--193.

\bibitem{Clark}
{\sc Clark, P.~L.}
\newblock On the {H}asse principle for {S}himura curves.
\newblock {\em Israel J. Math. 171\/} (2009), 349--365.

\bibitem{Drinfeld}
{\sc Drinfeld, V.~G.}
\newblock Elliptic modules.
\newblock {\em Mat. Sb. (N.S.) 94\/} (1974), 594--627.

\bibitem{GardeynJNT02}
{\sc Gardeyn, F.}
\newblock A {G}alois criterion for good reduction of {$\tau$}-sheaves.
\newblock {\em J. Number Theory 97}, 2 (2002), 447--471.

\bibitem{GardeynJNT03}
{\sc Gardeyn, F.}
\newblock The structure of analytic {$\tau$}-sheaves.
\newblock {\em J. Number Theory 100}, 2 (2003), 332--362.

\bibitem{GekelerFDM}
{\sc Gekeler, E.-U.}
\newblock On finite {D}rinfeld modules.
\newblock {\em J. Algebra 141}, 1 (1991), 187--203.

\bibitem{Goss}
{\sc Goss, D.}
\newblock {\em Basic structures of function field arithmetic}, vol.~35 of {\em
  Ergebnisse der Mathematik und ihrer Grenzgebiete (3)}.
\newblock Springer-Verlag, Berlin, 1996.

\bibitem{Hausberger}
{\sc Hausberger, T.}
\newblock Uniformisation des vari\'{e}t\'{e}s de {L}aumon-{R}apoport-{S}tuhler
  et conjecture de {D}rinfeld-{C}arayol.
\newblock {\em Ann. Inst. Fourier (Grenoble) 55}, 4 (2005), 1285--1371.

\bibitem{HayesCFT}
{\sc Hayes, D.~R.}
\newblock Explicit class field theory in global function fields.
\newblock In {\em Studies in algebra and number theory}, vol.~6 of {\em Adv. in
  Math. Suppl. Stud.} Academic Press, New York-London, 1979, pp.~173--217.

\bibitem{Jordan}
{\sc Jordan, B.~W.}
\newblock Points on {S}himura curves rational over number fields.
\newblock {\em J. Reine Angew. Math. 371\/} (1986), 92--114.

\bibitem{JLLocal}
{\sc Jordan, B.~W., and Livn\'{e}, R.~A.}
\newblock Local {D}iophantine properties of {S}himura curves.
\newblock {\em Math. Ann. 270}, 2 (1985), 235--248.

\bibitem{Lafforgue}
{\sc Lafforgue, L.}
\newblock Chtoucas de {D}rinfeld et conjecture de {R}amanujan-{P}etersson.
\newblock {\em Ast\'erisque}, 243 (1997), ii+329.

\bibitem{LaumonCDV}
{\sc Laumon, G.}
\newblock {\em Cohomology of {D}rinfeld modular varieties. {P}art {I}}, vol.~41
  of {\em Cambridge Studies in Advanced Mathematics}.
\newblock Cambridge University Press, Cambridge, 1996.
\newblock Geometry, counting of points and local harmonic analysis.

\bibitem{LRS}
{\sc Laumon, G., Rapoport, M., and Stuhler, U.}
\newblock {${\sD}$}-elliptic sheaves and the {L}anglands correspondence.
\newblock {\em Invent. Math. 113}, 2 (1993), 217--338.

\bibitem{MazurIsog}
{\sc Mazur, B.}
\newblock Rational isogenies of prime degree (with an appendix by {D}.
  {G}oldfeld).
\newblock {\em Invent. Math. 44}, 2 (1978), 129--162.

\bibitem{PapGenus}
{\sc Papikian, M.}
\newblock Genus formula for modular curves of {$\sD$}-elliptic sheaves.
\newblock {\em Arch. Math. (Basel) 92}, 3 (2009), 237--250.

\bibitem{PapCrelle1}
{\sc Papikian, M.}
\newblock Modular varieties of {$\mathscr{D}$}-elliptic sheaves and the
  {W}eil-{D}eligne bound.
\newblock {\em J. Reine Angew. Math. 626\/} (2009), 115--134.

\bibitem{PapHE}
{\sc Papikian, M.}
\newblock On hyperelliptic modular curves over function fields.
\newblock {\em Arch. Math. (Basel) 92}, 4 (2009), 291--302.

\bibitem{PapikianMZ}
{\sc Papikian, M.}
\newblock Endomorphisms of exceptional {$\sD$}-elliptic sheaves.
\newblock {\em Math. Z. 266}, 2 (2010), 407--423.

\bibitem{PapikianCrelle2}
{\sc Papikian, M.}
\newblock Local {D}iophantine properties of modular curves of
  {$\mathscr{D}$}-elliptic sheaves.
\newblock {\em J. Reine Angew. Math. 664\/} (2012), 115--140.

\bibitem{PapJNTHayes}
{\sc Papikian, M.}
\newblock {$\sD$}-elliptic sheaves and odd {J}acobians.
\newblock {\em J. Number Theory 133}, 3 (2013), 1012--1026.

\bibitem{PapRMS}
{\sc Papikian, M.}
\newblock Drinfeld-{S}tuhler modules.
\newblock {\em Res. Math. Sci. 5}, 4 (2018), Paper No. 40, 33.

\bibitem{PV}
{\sc Poonen, B., and Voloch, J.~F.}
\newblock The {B}rauer-{M}anin obstruction for subvarieties of abelian
  varieties over function fields.
\newblock {\em Ann. of Math. (2) 171}, 1 (2010), 511--532.

\bibitem{Reiner}
{\sc Reiner, I.}
\newblock {\em Maximal orders}, vol.~28 of {\em London Mathematical Society
  Monographs. New Series}.
\newblock The Clarendon Press, Oxford University Press, Oxford, 2003.
\newblock Corrected reprint of the 1975 original, With a foreword by M. J.
  Taylor.

\bibitem{Rosen}
{\sc Rosen, M.}
\newblock {\em Number theory in function fields}, vol.~210 of {\em Graduate
  Texts in Mathematics}.
\newblock Springer-Verlag, New York, 2002.

\bibitem{SerreInventiones}
{\sc Serre, J.-P.}
\newblock Propri\'{e}t\'{e}s galoisiennes des points d'ordre fini des courbes
  elliptiques.
\newblock {\em Invent. Math. 15}, 4 (1972), 259--331.

\bibitem{SerreLF}
{\sc Serre, J.-P.}
\newblock {\em Local fields}, vol.~67 of {\em Graduate Texts in Mathematics}.
\newblock Springer-Verlag, New York-Berlin, 1979.

\bibitem{TaelmanPhD}
{\sc Taelman, L.}
\newblock {\em On $t$-motifs}.
\newblock 2007.
\newblock Thesis (Ph.D.)--The University of Groningen.

\bibitem{Taguchi}
{\sc Taguchi, Y.}
\newblock The {T}ate conjecture for {$t$}-motives.
\newblock {\em Proc. Amer. Math. Soc. 123}, 11 (1995), 3285--3287.

\bibitem{Takahashi}
{\sc Takahashi, T.}
\newblock Good reduction of elliptic modules.
\newblock {\em J. Math. Soc. Japan 34}, 3 (1982), 475--487.

\bibitem{TatePDiv}
{\sc Tate, J.~T.}
\newblock {$p$}-divisible groups.
\newblock In {\em Proc. {C}onf. {L}ocal {F}ields ({D}riebergen, 1966)}.
  Springer, Berlin, 1967, pp.~158--183.

\end{thebibliography}

\end{document}